\documentclass{amsart}
\usepackage{setspace}
\usepackage{amsmath, amssymb,amsthm, amsfonts}
\newtheorem{thm}{Theorem}[section]
\newtheorem{lemma}[thm]{Lemma}
\newtheorem{prop}[thm]{Proposition}

\newcommand{\tr}{\mbox{tr}}
\newcommand{\Div}{\mbox{div}}
\renewcommand{\div}{\mbox{div}}

\newcommand{\Ric}{\mbox{Ric}}
\newcommand{\R}{\mathbb R}
\theoremstyle{remark}
\newtheorem{remark}[thm]{Remark}

\newtheorem{example}[thm]{Example}
\numberwithin{equation}{section}

\newcommand{\be}{\begin{equation}}
\newcommand{\ee}{\end{equation}}
\def\p{\partial}
\def\M{\mathcal{M}}
\def\la{\langle}
\def\ra{\rangle}
\def\lf{\left}
\def\ri{\right}
\def\Pi{\displaystyle{\mathbb{II}}}

\begin{document}
\title{Deformation of Scalar Curvature and Volume} 

\author{Justin Corvino}
\address{Department of Mathematics, Lafayette College, Easton, PA 18042, USA}
\email{corvinoj@lafayette.edu}
\author{Michael Eichmair}
\address{Department of Mathematics, ETH Z\"urich, 8092 Z\"urich, Switzerland}
\email{michael.eichmair@math.ethz.ch}
\author{Pengzi Miao}
\address{School of Mathematical Sciences, Monash University, Victoria 3800, Australia}
\email{Pengzi.Miao@sci.monash.edu.au}
\curraddr{Department of Mathematics, University of Miami, Coral Gables, FL 33146, USA}
\email{pengzim@math.miami.edu}

\thanks{The first author was partially supported by the NSF through grants DMS-0707317 and DMS-1207844, and by a Simons Foundation Collaboration Grant.  The second author was partially supported by the NSF through grant DMS-0906038 and by the SNF through grant 2-77348-12.
The third author was partially supported by the ARC through grant  DP0987650 and by a 2011 Provost Research Award of the University of Miami.}
\subjclass[2010]{Primary 53C21}

\begin{abstract}
The stationary points of the total scalar curvature functional on the space of unit volume metrics on a given closed manifold are known to be precisely the Einstein metrics. One may consider the modified problem of finding stationary points for the volume functional on the space of metrics whose scalar curvature is equal to a given constant. In this paper, we localize a condition satisfied by such stationary points to smooth bounded domains. The condition involves a generalization of the \emph{static} equations, and we interpret solutions (and their boundary values) of this equation variationally.  On domains carrying a metric that does not satisfy the condition, we establish a local deformation theorem that allows one to achieve simultaneously small prescribed changes of the scalar curvature and of the volume by a compactly supported variation of the metric. We apply this result to obtain a localized gluing theorem for constant scalar curvature metrics in which the total volume is preserved. Finally, we note that starting from a counterexample of Min-Oo's conjecture such as that of Brendle-Marques-Neves, counterexamples of arbitrarily large volume and different topological types can be constructed.  
\end{abstract}

\maketitle

\section{Introduction} 

Let $M$ be a closed manifold with dimension at least three, $\mathcal{M}$ the cone of Riemannian metrics on $M$, and $\mathcal M^c\subset \mathcal{M}$ the subset of Riemannian metrics with constant scalar curvature $c$. Let $V(g)=\mbox{vol}(M,g)$ be the volume of a metric $g \in \mathcal{M}$, and let $R(g)$ be its scalar curvature.  For $c\neq 0$, critical points of the restricted volume map $V_c:\mathcal M^c\rightarrow (0,\infty)$ are precisely stationary points of the total scalar curvature $\mathcal{R}(g)=\int_M R(g)\; d\mu_g$ restricted to $\mathcal M^c$. (Note that the total scalar curvature is a topological invariant in dimension two.) Critical metrics for $V_c$ are special, as they admit non-trivial solutions $(f,\kappa)$ to the overdetermined-elliptic system $L_g^*f=\kappa g$. Here, $L_g$ is the linearization of the scalar curvature operator, $L_g^*$ is its formal adjoint, and $\kappa$ is a constant. We make this precise in Theorem \ref{thm-var}. 

In this paper, we localize the above analysis to the case where the metric deformations are supported on the closure of a bounded domain $\Omega\subset M$.  In Theorem \ref{locdef}, we show that when the metric $g$ does not admit non-trivial solutions to $L_g^*f=\kappa g$, then one can achieve simultaneously  a prescribed perturbation of the scalar curvature that is compactly supported in $\Omega$ and a prescribed perturbation of the volume by a small deformation of the metric in $\overline{\Omega}$.  

The obstruction to finding such a deformation of the metric is the existence of a non-trivial solution $(f,\kappa)$ of the system $L_g^*f =\kappa g$ on $\Omega$. If such a non-trivial solution $(f,\kappa)$ exists, we call the metric $g$ \emph{$V$-static} with \emph{$V$-static potential} $f$.  This condition is a mild generalization of the \emph{static} equation $L_g^* f=0$, cf. \cite{cor:schw}. A metric $g$ is called \emph{static} if the static equation admits a non-trivial solution $f$, in which case $f$ is called a \emph{static potential} for $g$.  In Theorem \ref{thm-var}, we provide a variational characterization of $V$-static metrics, emphasizing the role of the boundary values of a $V$-static potential.  The case where $\kappa \neq 0$ and the $V$-static potential vanishes on the boundary was studied in \cite{MiaoTam08}, where an interesting volume comparison result (stated here as Theorem \ref{thm-vcomparison-1}) was proved. We include a new proof of this result from \cite{MiaoTam08} that actually leads to a slightly stronger result. 

We now give a precise statement of the local deformation theorem. Let  $h$ be a symmetric $(0,2)$-tensor on $M$.  The linearization $L_g$ of the scalar curvature map $R:\mathcal{M}\rightarrow \mathbb{R}$ is $L_g(h)= -\Delta_g(\tr_gh)+\Div_g \Div_g h - h\cdot \Ric(g)$, and its formal $L^2$-adjoint is $L_g^*f=-(\Delta_g f)g+\nabla_g^2f-f\Ric(g)$. (Our convention is that $\Delta_g f = \text{tr}_g(\nabla^2_g f)$.) The variation of the volume map $V:\mathcal{M}\rightarrow (0,\infty)$ is  $DV_g(h)= \frac{1}{2} \int_M \tr_g(h)\; d\mu_g$.  Let $\Theta(g):=(R(g), V(g))$.  Let $\mathcal{S}_g(h)=D\Theta_g(h)=(L_g(h), DV_g(h))$.  Its formal adjoint is then $\mathcal{S}_g^*(f,a)= L_g^*f + \frac{a}{2} g$. Thus $V$-static potentials correspond precisely to non-trivial elements in the kernel of $\mathcal{S}_g^*$.  
 
The Banach spaces $\mathcal{B}_0 = \mathcal{B}_0 (\Omega) \subset C^{0, \alpha} (\text{Sym}^2(T^*\Omega)) \times \R$ and $\mathcal{B}_2 = \mathcal{B}_2 (\Omega) \subset C^{2, \alpha} (\text{Sym}^2(T^*\Omega))$ and their respective norms $||\cdot||_0$ and $||\cdot||_2$ that appear in the statement of the following theorem are introduced in Section  \ref{subsection:pointwise}. 

\begin{thm} \label{locdef} Let $(\overline \Omega, g)$ be a compact $C^{4, \alpha}$ Riemannian manifold \footnote{Given an integer $k \geq 1$ and $\alpha \in (0, 1)$, a $C^{k, \alpha}$ Riemannian manifold $(M,g)$ consists of a smooth manifold $M$, possibly with non-empty boundary, and a tensor field $g \in C^{k, \alpha} (\text{Sym}^2(T^*M))$ that is everywhere positive definite.} of dimension $n \geq 2$ with boundary. Let $\Omega$ be the manifold interior of $\overline \Omega$. Assume that the equation $\mathcal{S}^*_{g} (f, a) = 0$ has no non-trivial solutions $(f, a) \in{C}^2 ({\Omega}) \times \R$. There exist $\epsilon, C >0$ so that for any $(\sigma, \tau) \in \mathcal{B}_0$ with $||(\sigma, \tau)||_0 < \epsilon$ there is a metric $\gamma$ on $\overline \Omega$ so that $R(\gamma)=R(g)+ \sigma $, $V(\gamma)=V(g)+\tau$. In fact, $\gamma - g \in \mathcal{B}_2$ and $\|\gamma-g\|_{2} \leq C \|(\sigma, \tau)\|_{0}$. In particular, $\gamma - g$ can be extended by $0$ as a $C^{2, \alpha}$ tensor across the boundary of $\overline \Omega$.  
\end{thm} 

The following version of Theorem \ref{locdef} includes the dependence on the metric and higher order regularity:

\begin{thm} \label{locdefTake2} Let $k \geq 4$. Let $(\overline \Omega, g_0)$ be a compact $C^{k, \alpha}$ Riemannian manifold of dimension $n \geq 2$ with boundary, and let $\Omega$ be the manifold interior of $\overline \Omega$. Assume that the equation $\mathcal{S}^*_{g_0} (f, a) = 0$ has no non-trivial solutions $(f, a) \in{C}^2 (\Omega) \times \R$. Let $\Omega_0 \subset \Omega$ be a non-empty open set that is compactly contained in $\Omega$. There exists an open neighborhood $U$ of $g_0$ in $C^{k, \alpha}(\overline \Omega)$ and $\epsilon, C>0$ such that for any $g \in U$, $\tau \in \mathbb{R}$, and $\sigma \in C^{k-4, \alpha}(\overline \Omega)$ with support in $\Omega_0$ and with $|\tau| + ||\sigma||_{C^{k-4, \alpha}} < \epsilon$, there is a $C^{k-2, \alpha}$ metric $\gamma$ on $\overline \Omega$ so that $\mathrm{supp}(\gamma - g)$ is compactly contained in $\Omega$, such that $||\gamma - g|| _{C^{k-2, \alpha}} \leq C \left( |\tau| + ||\sigma||_{C^{k-4, \alpha}}\right)$, and such that  $R(\gamma)=R(g)+ \sigma $, $V(\gamma)=V(g)+\tau$.  If $g$ and $\sigma$ are smooth, we can arrange for $\gamma$ to be smooth as well.  \end{thm} 

Now we present several examples to illustrate the static and $V$-static conditions.  Our convention is that $\lambda$ is an eigenvalue for a Schr\"odinger operator $T=\Delta_g + w$ with non-zero eigenfunction $u$ if $T(u)= -\lambda u$. 

\begin{example} \label{example_Einstein} Recall that a metric $g$ is critical for $V_{-1}$ on a closed manifold $M$ precisely when $g$ is Einstein with $R(g)=-1$.  We show that this is equivalent to being $V$-static with $R(g)=-1$.  Indeed, in the Einstein case, $(f, \kappa)=(1, 1/n)$ satisfies the critical equation $L_{g}^* f = \kappa g$.  Conversely, if $R(g)=-1$ and if $(M,g)$ admits a non-trivial solution $(f,\kappa)$ of $L_g^*f=\kappa g$, we first obtain that $-(n-1)\Delta_g f + f = n\kappa$ by taking the trace. Since $((n-1)\Delta_g - 1)$ is invertible, we conclude that $\kappa\neq 0$ and that $f=n\kappa$ is a solution. Plugging this back into the critical equation we obtain that  $-n\kappa \Ric(g)= \kappa g$. Thus $g$ is Einstein.  

The same argument shows that  when $R(g)=1$ and $\frac{1}{n-1}$ is not an eigenvalue of the Laplacian, then $g$ is $V$-static if and only if $g$ is Einstein.  \end{example}

\begin{example} \label{example_Ricciflat} A scalar-flat $V$-static metric $g$ on a closed manifold $M$ is Ricci-flat. To see this, let $(f, \kappa)$ be a non-trivial solution of $L^*_g f = \kappa g$. Taking the trace of this equation, we see that $\kappa = 0$ and that $f$ is equal to a non-zero constant. Using this information in the equation $L^*_g f = \kappa g$, we obtain that $(M, g)$ is Ricci-flat. Conversely, every Ricci-flat metric on $M$ is $V$-static and the space of solutions $(f, \kappa)$ of $L^*_g f = \kappa g$ is spanned by $(1, 0)$. Scaling changes the volume of a Ricci-flat metric at a nonzero rate while leaving the metric Ricci-flat. Therefore a Ricci-flat metric cannot be critical for $V_0$. 
\end{example}

\begin{example} \label{example:staticnotcritical} Consider the metric $g = (n-2)^{-1}g_{\mathbb{S}^1} + g_{\mathbb{S}^{n-1}}$ on $M = \mathbb{S}^1\times \mathbb{S}^{n-1}$ where $n \geq 3$. Then $f(t, \omega)=\sin(t)$ is a static potential for $g$. Clearly, scaling the $\mathbb{S}^1$-factor preserves the scalar curvature while the total volume changes. Thus $g$ is not a critical point for the volume functional on $\mathcal{M}^{(n-1)(n-2)}$. 
\end{example}

To summarize the above discussion, let $\mathcal{K}$ be the space of $V$-static metrics $g$ on a closed connected manifold $M$ of dimension at least three. 
This space contains all Einstein metrics and all metrics that are static. By Theorem A in \cite{ob}, a metric which is Einstein and static is either Ricci-flat or a round sphere. We can write $\mathcal K$ as a disjoint union $\mathcal{K}=\mathcal{K}_+\cup \mathcal{K}_0 \cup \mathcal{K}_-$ according to the sign of the constant scalar curvature $R(g)=c$, cf. Proposition \ref{prop:csc}. By Example \ref{example_Einstein}, the space $\mathcal{K}_-$ consists precisely of the Einstein metrics of negative scalar curvature. None of these metrics is static. By Example \ref{example_Ricciflat}, $\mathcal{K}_0$ is the space of Ricci-flat metrics, all of which are static, and none of which are critical for $V_0$.  The structure of $\mathcal{K}_+$ is more complicated.  $\mathcal{K}_+$ consists of metrics that are critical for $V_c$, e.g. the Einstein metrics of positive scalar curvature, and static metrics that are not critical for $V_c$, cf. with Example \ref{example:staticnotcritical}.  Static metrics in $\mathcal{K}_+$ admit $\frac{c}{n-1}$ in the spectrum of the Laplacian, such as the sphere  ($c=n(n-1)$) and $\mathbb{S}^1\times \mathbb{S}^{n-1}$ ($c=(n-1)(n-2)$). Further examples have been found by Kobayashi and Lafontaine in \cite{sh}.  If $\frac{c}{n-1}$ is not in the spectrum of the Laplacian of a metric in $\mathcal K_+$, then the metric is Einstein and non-static, for example $\mathbb {RP}^n$.


As our first application of Theorems \ref{locdef} and Theorem \ref{locdefTake2}, we establish a gluing result which is largely inspired by those in \cite{cip,imp,imp:flds}.  The result gives a condition that guarantees that two metrics with the same constant scalar curvature can be glued together to produce a metric with the same constant scalar curvature, preserving both the total volume and the original metrics outside a specified region.  
\begin{thm} \label{glue}
Fix  $n\geq 3$ and $k \geq 4$. Let $\sigma_n\in \{ -n(n-1), 0, n(n-1)\}$.  Let $(M_1, g_1)$ and $(M_2, g_2)$ be two compact $C^{k, \alpha}$ Riemannian manifolds such that $R(g_1)=\sigma_n=R(g_2)$. Assume that each $(M_i, g_i)$ contains a non-empty smooth domain $U_i\subset \mathrm{int}(M_i)$ where $g_i$ is not $V$-static. There exists a $C^{k-2, \alpha}$ metric $g$ on the connected sum $M_1\# M_2\supset (M_1\setminus U_1)\sqcup (M_2\setminus U_2)$ such that $R(g)=\sigma _n$, $\mathrm{vol}(M_1\# M_2, g)= \mathrm{vol}(M_1,g_1)+\mathrm{vol}(M_2,g_2)$, and $g= g_i$ on $M_i\setminus U_i$, $i=1,2$.  If $(M_1, g_1)$ and $(M_2, g_2)$ are smooth, then we can find $(M_1\# M_2, g)$ smooth with these properties. 
\end{thm}

There are many gluing results for constant scalar curvature and, more generally, the Einstein constraint equations in the literature. Gromov-Lawson \cite{gromlaw} and Schoen-Yau \cite{sy:psc} used different methods to prove that the existence of a positive scalar curvature metric on a manifold is preserved under surgeries of co-dimension at least three.  The seminal paper of Schoen \cite{singyam} on the singular Yamabe problem on the sphere has inspired a large number of works on scalar curvature gluing constructions. The resolution of the Yamabe problem shows that the connected sum of two closed manifolds admits a metric of constant scalar curvature in the conformal class of any metric on the sum. It is interesting and important  to understand in what way the constant scalar curvature metric on the sum can be made to reflect the geometry of the original summands.  Joyce \cite{joyce} produced constant scalar curvature metrics on connected sums of closed manifolds by constructing approximate solutions on the joined manifolds by hand, and then solving for a conformal deformation to constant scalar curvature. He also described the geometry of the resulting configuration.  A difference in Theorem \ref{glue} (as in \cite[Theorem 1.2]{cip}) is that we use a deformation out of the conformal class to preserve the initial metrics away from the gluing region. In particular, we note that in \cite{joyce}, the resulting metric on the connected sum of two zero scalar curvature metrics has constant negative scalar curvature.  

The conformal part of the proof of Theorem \ref{glue} follows closely the works \cite{imp, imp:flds}, see also \cite{mpu}, on gluing constructions for the Einstein constraint equations.  An important observation for localized gluing was made by Chru\'sciel-Delay \cite{cd}. They noticed that the conformal constructions could be combined with the localized deformation technique of Corvino-Schoen \cite{cor:schw, cs:ak} to produce, under certain non-degeneracy conditions, solutions to the Einstein constraint equations on connected sums for which the original data is left unchanged outside the gluing region. We refer to \cite[Theorem 1.2]{cip} for an analogue of our Theorem \ref{glue} in the case $\sigma_n\leq 0$.  A gluing construction for constant positive scalar metrics was obtained by Chru\'sciel-Pacard-Pollack \cite{cpp}.  An overview of these constructions with additional references is given in \cite[Sections 5.2-5.3]{corpol}. We also refer the reader to the recent work of Delay \cite{delay}.  

In the final section of this paper, we note how connect-sum constructions for scalar curvature can be combined with the recent counterexample to Min-Oo's conjecture
by Brendle, Marques, and Neves \cite{brendle-marques-neves} to produce counterexamples of different topological types and of large volume. Such examples are interesting in light of the recent results in \cite{MiaoTamMinOo}.

\subsection{Acknowledgments} The authors would like to thank S. Brendle, R. Mazzeo, D. Pollack, R. M. Schoen and L.-F. Tam for useful discussions on various aspects of this work. 


 \section{Variational characterization of $V$-static metrics}  
 
Let $(\overline \Omega, g)$ be a connected  $n$-dimensional compact $C^3$ Riemannian manifold with boundary, and let $\Omega$ be the manifold interior of $\overline \Omega$. We say that $(\overline \Omega, g)$ is \emph{$V$-static} (or simply that the metric $g$ is $V$-static) if the equation \begin{equation} \label{eqn:Vstatic} \mathcal{S}_g^*(f, a) = 0 \text{ on } \Omega \end{equation} admits a non-trivial weak solution $(f, a) \in H^1_{\mathrm{loc}} (\Omega) \times \R$;  $f$ is then called a \emph{$V$-static potential}. We will see in Proposition \ref{prop-bdryexten} that every solution $(f, a) \in H^1_{\mathrm{loc}} (\Omega) \times \R$ of (\ref{eqn:Vstatic}) is actually in $C^2(\overline \Omega) \times \R$. The goal of this section is to study properties of $V$-static metrics and to characterize the boundary values of $V$-static potentials. 


\subsection{The kernel of $\mathcal S_g^*$}   \label{sec:ker} 
The equation 
\begin{eqnarray} \label{eqn:statickappa} 
L_g^* f = \kappa g \text{ on } \Omega\end{eqnarray} 
is equivalent to (\ref{eqn:Vstatic}) with $a = - 2 \kappa$. Given $\kappa \in \R$, $L_g^*f = \kappa g$ is an overdetermined elliptic system for $f$. It is well-known how to re-cast (\ref{eqn:statickappa}) into a proper elliptic system for $(f,g)$ in appropriate coordinates (e.g. harmonic coordinates), cf. \cite{be} or \cite[p. 145-146]{cor:schw}.  In such coordinates, then, $f$ and $g$ are analytic. It follows that if $(\overline \Omega, g)$ is $V$-static, then so is any subdomain (with the restricted metric); this also follows from the proof of Proposition \ref{prop:csc}. 
 
 The following property of $V$-static metrics follows as in \cite[Proposition 2.3]{cor:schw}, see also \cite[Theorem 7 (i)]{MiaoTam08}.

\begin{prop} \label{prop:csc}
Assume that for some constant $\kappa \in \R$ there exists a non-trivial weak solution $f \in H^{1}_{\mathrm{loc}} (\Omega)$ of  \eqref{eqn:statickappa}. Then $g$ has constant scalar curvature. 
\end{prop}
 \begin{proof} By elliptic regularity, $f \in C^2(\Omega) \cap H^3_{\mathrm{loc}} (\Omega)$. Taking the divergence of the equation $L_g^*f=\kappa g$ and using the Bianchi identity and the Ricci formula, it follows that $fdR(g)=0$. Along a unit-speed geodesic $\gamma$ with $\gamma (0) = p$, the equation $L_g^*f=\kappa g$ reduces to a second-order ODE with initial data $(f(p), df (\gamma'(0)))$.  Indeed, if $h(t)=f(\gamma(t))$, then $$h''(t)= \nabla^2_g f (\gamma'(t), \gamma'(t)) = \Big(\Ric(g)(\gamma'(t), \gamma'(t)) - \frac{1}{n-1}\Big) h(t) - \frac{\kappa}{n-1}.$$ In the homogeneous case $\kappa=0$, observe that if $f$ has a zero that is a critical point, then $f$ is identically zero.  Thus the zero set of $f$ has codimension one. It follows that $d R(g) = 0$ so that the scalar curvature $R(g)$ is constant.  If $\kappa\neq 0$, then a solution of the (inhomogeneous) ODE cannot vanish identically in a non-empty open set, from which we can again conclude that $R(g)$ is constant.  \end{proof} 
 
The ODE argument in the proof of Proposition \ref{prop:csc} shows that the kernel of $L_g^*$ has dimension at most $(n+1)$.  Thus, the dimension of the kernel of $\mathcal{S}_g^*$ is at most $(n+2)$. This maximal dimension is achieved, for example, by the standard metric on the sphere $\mathbb S^n$. Viewing $\mathbb S^n$ as the unit sphere in $\R^{n+1}$ with center at the origin, the kernel is spanned by $(x^j|_{\mathbb S^n}, 0)$, $j=1, \ldots, n+1$, and $(1, 2(n-1))$ in this case. 
 
By employing the exponential map from points near the boundary and using basic facts about existence, uniqueness, and dependence on initial data for ODEs as in \cite[Proposition 2.5]{cor:schw}, we see that every solution $f$ of \eqref{eqn:statickappa} extends to the boundary as a $C^2$ function; using an interior elliptic estimate, or appealing to the finite-dimensionality of the kernel, we also obtain an estimate on such solutions:

\begin{prop}  \label{prop-bdryexten}
Every weak solution $ f  \in  H^{1}_{\mathrm{loc}} (\Omega)$ of \eqref{eqn:statickappa} is actually in $C^2(\overline \Omega)$. There is a constant $C=C(\Omega, g)$ so that for $\epsilon >0$ sufficiently small, and for any solution $f$ of \eqref{eqn:statickappa}, $||f||_{C^2(\overline \Omega)} \leq C ||f||_{H^1(\Omega_\epsilon)}$, where $\Omega_\epsilon  =\{x \in \Omega : d(x,\partial \Omega) > \epsilon\}$.
\end{prop}


\subsection{The variational principle} In this section, we characterize the boundary values of solutions $f$ of \eqref{eqn:statickappa} whose existence is ensured by Proposition \ref{prop-bdryexten}. For simplicity, we will assume that $n \geq 3$ and that $(\overline \Omega, g)$ is smooth in this subsection and the next. 

We adopt the notation from \cite{MiaoTam08}. Let $\gamma$ be a smooth Riemannian metric on $\p \Omega  $.  Let $c$ be a constant.  For any 
integer $ k > \frac{n}{2} + 2 $, let $ \mathcal{M}^c_\gamma $ denote the set of $ H^k$ Riemannian metrics $ g $ on $ \overline \Omega $ such that 
$R ( g ) = c \ \ \mathrm{and} \ \ g |_{T(\p \Omega) } = \gamma,$ where $ R(g) $ is the scalar curvature of $ g $ and $ g |_{T( \p \Omega) } $ is the metric induced by $g$ on $\partial \Omega$.  We recall from \cite{MiaoTam08} that if $g$ is such that $\Delta_g + \frac{c}{n-1}$ has positive (Dirichlet) spectrum, then $\mathcal{M}^c_\gamma$ is a Hilbert manifold near $g$. Let $\nu$ be the outward unit normal to $\p \Omega$, let $\Pi(X,Y)= \la \nabla_X \nu, Y \ra$ for vector fields $X, Y$ tangent to $\partial \Omega$, and let $H=\mbox{tr}_{\gamma} (\Pi)$ be the mean curvature. (Our sign convention follows that of  \cite{MiaoTam08}.)  

The following theorem  provides a general context unifying \cite[Theorem 5]{MiaoTam08} and \cite[Theorem 2.1]{MiaoShiTam09}. 

\begin{thm} \label{thm-var}
Let $\kappa$ be a constant and let $\phi$ be a smooth function 
on $\partial \Omega$. We assume that either $ \kappa \neq 0 $ or 
that $ \phi $ does not vanish identically. 
Consider the functional  {on}  $\mathcal{M}^c_\gamma$ given by 
\be
g\mapsto E_{\kappa, \phi} ( g ) = \kappa V (g) - \int_{\p \Omega} H \phi \ d  \sigma \; ,
\ee  
where $ V(g) $ is the volume of $(\overline \Omega, g)$ and $ d  \sigma  $ is the volume form
of $ \gamma $.  Suppose $ g \in \M^c_\gamma $ is  a smooth metric such that
the operator
$  \Delta_g + \frac{c}{n-1}$ 
has positive (Dirichlet) spectrum. Then $ g $ is a critical point of 
$ E_{\kappa, \phi} ( \cdot ) $ on $ \mathcal{M}^c_\gamma$
if and only if there exists a smooth function $ f $ on $ \overline \Omega$ with
\be \label{eq-target}
L_g^* f=  \kappa g  \ \   \mathrm{ in } \ \ \Omega \ \ \mathrm{ and } \ \ f  =  \phi \ \ \mathrm{ on } \ \  \partial \Omega  .
\ee
\end{thm}
\begin{proof}
We follow the proof of Theorem 2.1 in \cite{MiaoShiTam09}. 
Let $ \{ g(t) \}_{| t | \le \epsilon } \subset \M^c_\gamma $ be a continuously differentiable path such that $ g (0) = g $. Let $h = g^\prime (0) $.  Let $H(t)$ be the mean curvature of $\partial \Omega$ in $(\overline \Omega, g(t))$ computed with respect to the outward unit normal as above. 
A calculation as in \cite[(34)]{MiaoTam08}
yields that 
\be \label{linearizedH}
2 H^\prime (0) = [ d ( \tr_ g h ) - \div_g h ] (\nu) - \div_\gamma X - \la \Pi, h \ra_\gamma 
\ee
where $ \nu $ is the outward unit normal to $ \p \Omega  $ in $(\Omega, g(0) )$,
$ X $ is the vector field dual to the $1$-form $ h( \nu, \cdot) |_{T(\p \Omega) } $ on $ (\p \Omega , \gamma)$,
$ \div_\gamma X $ is the divergence of $ X $ on $ (\p \Omega , \gamma)$, and $ \la \cdot, \cdot \ra_\gamma $ is the metric product on $(\p \Omega, \gamma)$. 
Using that $ h|_{T(\p \Omega)  } = 0$, it follows that
\be \label{eqn:linearizationE} 
2 \frac{d}{dt}\big |_{t =0}  E_{\kappa, \phi} ( g(t) ) = \int_\Omega \kappa \tr_g h 
- \int_{\p \Omega}  \phi \left\{ [ d ( \tr_ g h ) - \div_g h ] (\nu) - \div_\gamma X \right\}
\ee
where we have omitted the volume forms. For any function $f$ on $\Omega$ with $f = \phi$ along $\partial \Omega$, we can integrate by parts in (\ref{eqn:linearizationE}) to obtain
\begin{align} \label{linearE1} 
2   \frac{d}{dt} &\big|_{t =0}   E_{\kappa, \phi}  ( g(t) )  \\
= &   \int_\Omega \kappa \tr_g h  - f \lf[\Delta_g ( \tr_g h ) - \div_g ( \div_g h )\ri]   + ( \Delta_g f ) \tr_g h   
 - \la \nabla^2_g f, h\ra_g  \nonumber \\
 &  +  \int_{\p \Omega} 
 h( \nu, \nabla_g f) -  h(\nu, \nabla_{\gamma}  f )  - \tr_g h \frac{\p f}{\p \nu} \nonumber
\end{align}
where $ \la \cdot, \cdot \ra_g $ denotes the metric product on $ (\overline \Omega, g) $ and 
$ \nabla_{\gamma}  $ is the gradient operator on $(\p \Omega , \gamma)$.
Since $ h |_{ T (\p \Omega) } = 0 $,  
\be \label{bdryeq1}
h ( \nu, \nabla_g f) - h ( \nu, \nabla_{\gamma}  f) -  \tr_g h \frac{\p f}{\p \nu}  
= 0 \ \ \ \mathrm{on} \ \ \p \Omega. 
\ee
On the other hand, the fact that $  \{ g(t) \}_{| t | \le \epsilon } \subset \M^c_\gamma $ implies that
\be \label{linearR}
L_g(h)= -  \Delta_g ( \tr_g h) + \div_g \div_g h - \la h, \Ric(g) \ra_g = 0, 
 \ee
where we recall that $L_g(h)=D R_g ( h )$ is the linearization of the scalar curvature map at $ g $ in direction $h$.  Therefore, it follows from \eqref{linearE1}-\eqref{linearR} that
\begin{eqnarray}  \label{gradE} 
2 \frac{d}{dt} \big|_{t=0}  E_{\kappa, \phi}  ( g(t) ) &=&   \int_\Omega \la h,   f \Ric(g) + ( \Delta_g  f) g - \nabla^2_g f + \kappa g \ra_g  \\ \nonumber &=& \int_\Omega \la h,  - L^*_g f + \kappa g \ra_g = 0. 
\end{eqnarray}
Hence if $f$ is a solution of (\ref{eq-target}), then $g$ is a critical point of $ E_{\kappa, \phi} (\cdot) $ on $ \mathcal{M}^c_\gamma$.

For the other direction, assume now that $g$ is a critical point of $E_{\kappa, \phi} (\cdot)$, and consider the unique solution $f$ of the boundary value problem 
\be \label{bdryprob}
\left\{
\begin{array} {rcl}
( n - 1)  \Delta_g f   + c f  & = & - n \kappa \ \ \ \mathrm{in} \ \ \Omega  \\
 f  & = & \phi \ \  \ \  \ \ \mathrm{on} \ \ \p \Omega  .
\end{array} 
\right.
\ee 
Let $ \hat h $ be an arbitrary smooth symmetric $(0,2)$-tensor with compact support in $ \Omega $.
Since the first Dirichlet eigenvalue of $  \Delta_g + \frac{c}{n-1}$ is positive, by \cite[Proposition 1]{MiaoTam08} there exist $ t_0 > 0 $ and $ \epsilon > 0 $ such that,
 for every $ t \in ( - t_0, t_0) $, there exists a unique smooth positive function $u(t)$ on $ \overline \Omega $ with $ | u(t) - 1 | \le \epsilon $
 such that $u(t)=1$ on $ \p \Omega  $, such that
$g (t) = u(t)^{\frac{4}{ n-2} } (g + t \hat h) \in \mathcal{M}^c_\gamma$, and such that $ \{ u(t) \}_{ | t |< t_0 }$ is differentiable at $ t = 0 $ with $ u(0) = 1 $.
For such a path $  g(t) $, we have $ h  :=g'(0) = \frac{4}{n-2} u'(0) g + \hat h $. 
Hence, by \eqref{gradE} and the fact that $f$ is a solution to \eqref{bdryprob},
we have
\be \label{gradE-1}
0 =   \int_\Omega \la \hat{h},   f \Ric(g) + ( \Delta_g  f) g - \nabla^2_g f + \kappa g \ra_g  .
\ee
Since $ \hat{h}$ can be chosen arbitrarily, we conclude that $ f $ satisfies \eqref{eq-target}. 
\end{proof}

\subsection{A volume comparison result for $V$-static metrics}

When the function $ \phi $  in Theorem \ref{thm-var}   is chosen to be identically zero and $\kappa =1$, then 
Theorem \ref{thm-var} reduces to Theorem 5 in \cite{MiaoTam08} and claims that  for a  metric $ g \in \mathcal{M}^c_\gamma$ for which $\Delta_g + \frac{c}{n-1}$ has positive first Dirichlet eigenvalue, the system \be \label{eq-target-zero}
\left\{
\begin{array} {rcl}
- ( \Delta_g f ) g + \nabla^2_g f - f \Ric (g) & = &  g \ \ \ \mathrm{in} \ \ \Omega  \\
 f & = & 0 \ \ \  \mathrm{on} \ \ \p \Omega  
\end{array} 
\right.
\ee
admits a solution $ f \in {C}^2(\overline \Omega)$ if and only if $ g $ 
is a critical point of the volume functional $ V(\cdot)$ restricted to  $ \mathcal{M}^c_\gamma$. 
We recall the following volume comparison result from \cite{MiaoTam08} for such metrics when $c=0$. 

\begin{thm} [\cite{MiaoTam08}] \label{thm-vcomparison-1}
Let $ g $ be a smooth, scalar flat metric on $\overline  \Omega $.
Suppose there exists a function $ f $ such that $ g $ and $ f $ satisfy
\eqref{eq-target-zero}. 
Let $ \gamma $ be the metric induced on $ \Sigma  = \partial \Omega$.
Suppose $ \Sigma $ is connected and that
 $(\Sigma, \gamma)$ can be isometrically embedded in $\R^n$
as a compact strictly convex hypersurface $\Sigma_0$.
If $n>7$, where $n$ is the dimension of $\overline \Omega$, we assume in addition that $ \overline \Omega $ is  spin.
Then
$$  V(g) \ge V_0 $$
where $V(g)$ is the volume of $(\overline \Omega,g)$ and $V_0$ is the
Euclidean volume of the compact domain bounded by $\Sigma_0$
in $ \mathbb{R}^n $.
Moreover, $ V (g) = V_0 $ if and
only if $(\overline \Omega, g)$ is isometric to a standard ball in $\R^n$.
\end{thm}

The proof of Theorem \ref{thm-vcomparison-1} in \cite{MiaoTam08} uses the result of Shi and Tam in \cite{ShiTam02} and thus depends on the Positive Mass Theorem \cite{SchoenYau79, Witten81}.  Here we include another proof of Theorem \ref{thm-vcomparison-1} that does not depend on the Positive Mass Theorem, so we can omit the spin assumption in high dimensions. We start with the following proposition. 

\begin{prop} \label{propV}
Let $ g $ be a smooth, scalar flat metric  on $\overline \Omega $.
Suppose there exists a function $ f $ such that $ g $ and $ f $ satisfy
\eqref{eq-target-zero}. 
Let $ \gamma $ be the metric induced on $ \Sigma  = \partial \Omega$, let $ | \Sigma  | $ be the area of $(\Sigma, \gamma)$, and suppose that $ \Sigma $ is connected. Then
\begin{enumerate}
\item[a)] $ \int_\Sigma R_\gamma > 0 $,
 where $ R_\gamma $ is the scalar curvature of $ (\Sigma, \gamma) $.
\item[b)] The volume $V(g)$ of $ (\overline \Omega, g) $ satisfies
$$
V (g) \geq  \frac{ \sqrt{ ( n-2)(n-1)} }{n}  \lf( \int_\Sigma R_\gamma \ri)^{-\frac{1}{2} }  | \Sigma |^\frac32.
$$
Equality holds if and only if $ (\overline \Omega, g)$ is isometric to a standard ball in $ \R^n$.
\item[c)] When $ n = 3 $, one has
$$
V (g) \geq   \frac{ | \Sigma |^\frac32 } { 6 \sqrt \pi }.$$
Equality holds if and only if $ (\overline \Omega, g)$ is isometric to a round ball in $ \R^3$.
\end{enumerate}
\end{prop}

\begin{proof} 
Let $ \nu $ be the outward unit normal to $ \Sigma $. Let $ H $ and $ \Pi$ be the mean curvature and
 the second fundamental form of 
$ \Sigma $ in $(\overline \Omega, g)$ with respect to $ \nu$. By Theorem 7 (iii) in \cite{MiaoTam08}, 
 $ f $ is positive on $\Omega$, 
$ H $ is a positive constant, and $(n-1) \Pi = H \gamma $.
Moreover, by (48) and (53)  in \cite{MiaoTam08},  we have that 
$ H \frac{\partial f}{\partial \nu} = - 1 $  and  
$ | \Sigma | = \frac{n}{n-1} H V (g) $. Hence
\be \label{eq-npf-a}
\int_\Omega \la \nabla^2_g f , \Ric(g) \ra = 
\int_\Sigma \Ric(g) ( \nu, \nabla_g f ) - \int_\Omega
\la d f, \div_g \Ric(g) \ra 
=   \frac{\partial f}{\partial \nu} \int_\Sigma    \Ric(g)(\nu, \nu)  
\ee
where the first equality follows from an integration by parts, and where we used that $2 \div_g (\Ric(g)) =  d R(g) = 0$ and that  $\frac{\partial f}{\partial \nu} $ is constant along $\Sigma$ to justify the second equality. 
Taking the metric product of \eqref{eq-target-zero} with $\Ric(g)$ and using again that $R(g) = 0$, we see that
\be \label{eq-npf-b}
\la \nabla^2_g  f , \Ric(g) \ra = 
\la f \Ric(g) + ( \Delta_g f) g + g , \Ric(g) \ra 
=   f | \Ric(g) |^2 .
\ee
Equations \eqref{eq-npf-a} and \eqref{eq-npf-b}, together with the fact that $ H \frac{\partial f}{\partial \nu} = - 1$, give 
\be \label{eq-npf-r}
- \int_\Omega f | \Ric(g) |^2  
 = \frac{1}{H}  \int_\Sigma \Ric(g)(\nu, \nu) .
\ee
In particular, this shows that
\be \label{eq-npf-d}
  \int_\Sigma \Ric(g)(\nu, \nu)  \leq 0,
\ee
with equality if and only if $ \Ric(g) = 0 $ on $ \Omega$. 

The Gauss Equation, along with the fact $ R(g) =0 $ and $ \Pi = \frac{ H }{ n-1} \gamma$, implies 
\be \label{eq-npf-e}
2 \Ric(g) (\nu, \nu) = \frac{ n-2 }{ n -1 }   H^2 - R_\gamma .
\ee
It follows from \eqref{eq-npf-d} and \eqref{eq-npf-e} that
\be \label{eq-npf-f}
\int_\Sigma R_\gamma \geq  \ \frac{ n-2 }{ n -1 }  \int_\Sigma H^2 
= \frac{ n-2 }{ n -1 }  H^2 | \Sigma |.  
\ee
This proves a). 

The inequality in b) follows  from 
 \eqref{eq-npf-f} and the fact $ | \Sigma | = \frac{n}{n-1} H V (g) $. If equality holds, then $\Ric(g) = 0$ on $\Omega$. That $(\overline \Omega, g)$ is isometric to a ball in $\R^n$ in this case then follows from  \cite[Theorem 2.1]{MiaoTam-TAMS}.

Finally, c)  follows from b) and the Gauss-Bonnet Theorem. 
\end{proof}

The fact that Proposition \ref{propV} implies Theorem \ref{thm-vcomparison-1} was first noted by Tam \cite{Tam-private}.  We thank Luen-Fai Tam for  pointing out the following lemma. 

\begin{lemma} \label{lma-estofintR} 
Suppose $ \Sigma \subset \R^n $ is an embedded, closed, strictly convex hypersurface.
Let $ R_\gamma$ be the scalar curvature of $ \Sigma$ with respect to the metric $\gamma$ induced from the Euclidean metric, let $|\Sigma| $ be its area, and let $V$ be the Euclidean volume of the region enclosed by $ \Sigma$.
Then
\be \label{eq-estofintR}
\int_{\Sigma} R_\gamma \le \frac{(n-1)(n-2)|\Sigma|^3}{n^2V^2}.
\ee
\end{lemma}

\begin{proof}
Let $ H $ denote the (positive) mean curvature of $ \Sigma$.
Taking $ i$, $j$, $k$ to be $ 1$, $2$, $3$ and then $0$, $ 1$, $2$
in  (6.4.6) in \cite[p. 334]{Schneider},  one arrives at two of the 
Minkowski inequalities 
\be \label{wijk}
 (W_1)^2 \ge W_0 W_2 ,  \ (W_2)^2 \ge W_1 W_3,
\ee
where $W_0 = V$, $  W_1 = \frac{1}{n} | \Sigma |$,
$W_2 = \frac{1}{n(n-1)} \int_{\Sigma} H$, and
 $ W_3 = \frac{ 1}{n(n-1)(n-2)} \int_{\Sigma} R_\gamma .$
Clearly, \eqref{wijk} implies \eqref{eq-estofintR}. 
\end{proof}

Theorem \ref{thm-vcomparison-1} without the spin assumption in dimension $n >7$ now follows  from
Proposition \ref{propV} and Lemma \ref{lma-estofintR}.


\section{Proof of Theorem \ref{locdef}}

The proof of Theorem \ref{locdef} is similar to those of the localized deformation theorems in \cite{cor:schw, cs:ak, cd}. It proceeds by iteration with a linear correction at each stage. The linearized problem is solved variationally. This requires delicate weighted $L^2$-estimates. The pointwise bounds on these variational solutions required to establish convergence of the iteration follow from interior Schauder estimates. 

\subsection{Function spaces}  \label{func-sp} Let $k$ be a non-negative integer, $\alpha \in (0, 1)$, and let $(\overline \Omega, g)$ be a compact $C^{k, \alpha}$ Riemannian manifold with boundary. Let $\Omega$ denote the manifold interior of $\overline \Omega$. Let $\ell \leq k$ be a non-negative integer, and let $\rho$ be a positive measurable function on $\Omega$. Below, we use the connection and the tensor norms induced by $g$, and we integrate with respect to the volume form $d\mu_g$.  

Let $L^{2}_{\rho}(\Omega)$  be the set of  functions (or tensor fields) $u$ such that $|u|\rho^{1/2} \in L^{2}(\Omega)$ and let $\|u\|_{L^2_{\rho}(\Omega)}= \|u\rho^{1/2}\|_{L^2(\Omega)}$.  The pairing \[\langle u,v\rangle_{L^2_{\rho}(\Omega)} = \langle u\rho^{1/2},v\rho^{1/2}\rangle_{L^2(\Omega)}\] makes $L^2_{\rho}(\Omega)$ into a Hilbert space.  Let $H^{\ell}_{\rho}(\Omega)$ be the Hilbert space of $L^{2}_{\rho}(\Omega)$ functions (tensor fields) whose covariant derivatives up to and including order $\ell$ are also in $L^{2}_{\rho}(\Omega)$.   The inner product is defined by incorporating the $L^2_{\rho}(\Omega)$-pairings on all the derivatives, so that $$\| u\|_{H^\ell_{\rho}(\Omega)}^2= \sum\limits_{j=0}^\ell \|\nabla_g ^j u\|^2_{L^2_{\rho}(\Omega)}.$$  

Assume now that $k \geq 1$. Let $d(x)=d(x, \partial \Omega)$ be the distance to the boundary $\partial \Omega$ computed with respect to the metric $g$. Then $d(x)$ is a $C^{k, \alpha}$ function near $\partial \Omega$ \cite{Foote}. We will use a $C^{k, \alpha}$ weight $\rho$ with $0<\rho\leq 1$ on $\Omega$ with the following boundary behavior: $\rho$ depends monotonically on the distance $d$ to $\partial \Omega$, $\rho=e^{-1/d}$ near $\partial \Omega$, and $\rho \equiv 1$ outside a neighborhood of $\partial \Omega$.  Note that $|\nabla_g^{\ell}\rho|\leq C(\ell)d^{-2\ell}e^{-1/d}$. To be precise, we let $\rho(x)=\tilde{\rho}(d(x))$, where $\tilde{\rho}:\mathbb R \rightarrow [0,1]$ is smooth and monotone, $\tilde{\rho}'\geq 0$, with $\tilde{\rho}(t)>0$ for $t>0$, and $\tilde{\rho}(t)= e^{-1/t}$ on some interval $(0, d_0)$.  

Let $\phi>0$ be a $C^{k, \alpha}$ function on $\Omega$ such that for all $x\in \Omega$, $B(x, \phi(x))\subset \Omega$, and so that near $\partial \Omega$, $\phi=d^2$.  For $r, s\in \mathbb R$, let $\varphi=\phi^r  \rho^s$. For a $C^{\ell, \alpha}$ function $u : \Omega \to \R$ we define $\|u\|_{C^{\ell,\alpha}_{\phi, \varphi}(\Omega)}$ by
\begin{align*}
\sup\limits_{x\in \Omega}\Big( & \sum\limits_{j=0}^{\ell} \varphi (x) \phi(x)^j \|\nabla_g ^j u \|_{C^0(B(x, \phi(x)/2))}+ \varphi (x) \phi(x)^{\ell+\alpha} [\nabla_g^\ell u]_{0,\alpha;B(x, \phi(x)/2)}\Big).
\end{align*}
We let $C^{\ell,\alpha}_{\phi, \varphi}(\Omega)$ be the space of all functions $u\in C^{\ell,\alpha}(\Omega)$ for which $\|u\|_{C^{\ell,\alpha}_{\phi, \varphi}(\Omega)}<\infty$. Note that $\|\cdot \|_{C^{\ell,\alpha}_{\phi, \varphi}(\Omega)}$ is a Banach norm on this space. When the context is clear, we will  suppress the domain in the notation below. With our choice of $\phi$ and $\varphi$, we have that $\|u\|_{C^{\ell,\alpha}_{\phi, \varphi}}$ is equivalent to $\|u\varphi\|_{C^{\ell,\alpha}_{\phi, 1}}$. Moreover, differentiation is 
continuous as a map from $C^{\ell,\alpha}_{\phi, \varphi}(\Omega)$ to $C^{\ell-1,\alpha}_{\phi, \phi\varphi}(\Omega)$. 

These weighted H\"{o}lder norms are equivalent to those defined in \cite[p. 66]{cd}. 


\subsection{Coercivity estimate for $\mathcal S_g^*$} Let $\{ x\in \Omega: d(x)<\epsilon_0\}$ be a regular tubular neighborhood of $\partial \Omega$
and let $\Omega_{\epsilon}=\{ x\in \Omega: d(x)>\epsilon\}$. 

\begin{prop} [\protect{Cf. \cite[Theorem 3]{cor:schw}}] \label{prop:coest} Let $(\overline \Omega, g)$ be as in Theorem \ref{locdef}.  There exists a constant $C>0$ so that for all $(u, a)\in H^2_{\rho}(\Omega)\times \mathbb R$,
\be 
\label{eq:coest} 
\| (u,a)\|_{H^2_{\rho}(\Omega)\times \mathbb R} \leq C \|\mathcal S_g^*(u,a)\|_{L^2_{\rho}(\Omega)}.
\ee 
\end{prop}
\begin{proof} There is a constant $D>0$ so that for all $\epsilon>0$ sufficiently small, there is an extension operator $E_{\epsilon}: H^2(\Omega_{\epsilon})\rightarrow H^2(\Omega)$ with norm bounded by $D$. The equation $\mathcal S_g^*(u,a)=  -(\Delta_g u) g + \nabla_g^2 u - u \Ric(g) + \frac{a}{2}g$ shows that 
\be \label{eq:coest0}
\|(u,a)\|_{H^2(\Omega_{\epsilon})\times \R} \leq C(n, g, \Omega) \left( \|\mathcal S_g^*(u,a)\|_{L^2(\Omega_{\epsilon})} + \|(u,a)\|_{H^1(\Omega_{\epsilon})\times \R} \right).
\ee 
Note that $\mathcal S_g^*$ has trivial kernel in $H^1_{\rm{\mathrm{loc}}}(\Omega_{\epsilon})$ for $\epsilon >0$ sufficiently small. Indeed, the kernels $K_{\epsilon}$ of $\mathcal S_g^*$ on $\Omega_{\epsilon}$ decrease as $\epsilon \downarrow 0$ (by restriction, which is injective by the remarks following \eqref{eqn:statickappa}). Since each is at most $(n+2)$-dimensional (cf. Section \ref{sec:ker}) and there is no kernel on $\Omega$, they must stabilize at $\{0\}$.  

We claim that there is a constant $C>0$ so that for all $\epsilon>0$ sufficiently small and $(u,a)\in H^2(\Omega_{\epsilon})\times \R$, 
\be \label{coest-1}
\| (u,a)\|_{H^2(\Omega_{\epsilon})\times \mathbb R} \leq C \|\mathcal S_g^*(u,a)\|_{L^2(\Omega_{\epsilon})} .
\ee

We prove this by contradiction.  Suppose the estimate does not hold. There is a sequence $\epsilon_j>0$ with $\epsilon_j \downarrow 0$ and $(u_j,a_j)\in H^2(\Omega_{\epsilon_j})\times \R$ such that 
\be \label{eq:contra}
\| (u_j,a_j)\|_{H^2(\Omega_{\epsilon_j})\times \mathbb R} \geq j \|\mathcal S_g^*(u_j,a_j)\|_{L^2(\Omega_{\epsilon_j})} .
\ee
Let $\tilde u_j=E_{\epsilon_j}(u_j)$ be the extension of $u_j$ to $\Omega$. Then  
\be \label{eq:extest}
\|(u_j,a_j)\|_{H^2(\Omega_{\epsilon_j})\times \R} \leq \|(\tilde u_j, a_j)\|_{H^2(\Omega)\times \R} \leq D \|(u_j, a_j)\|_{H^2(\Omega_{\epsilon_j})\times \R}  .
\ee
We normalize so that $\|(\tilde u_j, a_j)\|_{H^1(\Omega)\times \R} =1$.  Using (\ref{eq:coest0}) and  (\ref{eq:contra}), we obtain
\begin{align*}
 \|(u_j, a_j)\|_{H^2(\Omega_{\epsilon_j})\times \R} 
& \leq  C(n, g, \Omega) \left( \|\mathcal S_g^*(u_j,a_j)\|_{L^2(\Omega_{\epsilon_j})} + \|(u_j,a_j)\|_{H^1(\Omega_{\epsilon_j})\times \R} \right)\\
& \leq C(n, g, \Omega) \left( j^{-1}\| (u_j,a_j)\|_{H^2(\Omega_{\epsilon_j})\times \mathbb R}+1\right).
\end{align*}
For $j$ large enough so that $C(n, g, \Omega) j^{-1} \leq \frac{1}{2}$, we obtain $ \|(u_j, a_j)\|_{H^2(\Omega_{\epsilon_j})\times \R} \leq 2 C(n, g, \Omega)$.  By (\ref{eq:extest}), we then have $ \|(\tilde u_j, a_j)\|_{H^2(\Omega)\times \R} \leq 2 D C(n, g, \Omega)$. By the Rellich Lemma and the fact that the sequence $\{ a_j\}$ is bounded, there exist $(u,a)\in  H^2(\Omega)\times \R$ and a subsequence of $\{ (\tilde u_j, a_j)\}$ that converges to $(u,a)$ weakly in $H^2(\Omega)\times \R$ and strongly in $H^1(\Omega)\times \R$.  The latter implies that $\|(u,a)\|_{H^1(\Omega)\times \R}=1$.  Moreover, by pairing $\mathcal S_g^*(u,a)$ in $L^2(\Omega)$ with $h\in C^{2}_c(\Omega)$, and using (\ref{eq:contra}), we see $\mathcal S_g^*(u,a)=0$ holds weakly, so that $(u,a)$ is a non-trivial element of the kernel of $\mathcal S_g^*$.  This is a contradiction. Thus (\ref{coest-1}) holds uniformly for $\epsilon>0$ small, as asserted. 

The uniformity of (\ref{coest-1}) in $\epsilon>0$ allows us to promote this estimate to the weighted coercivity estimate (\ref{eq:coest}) exactly as in \cite[p. 149-150]{cor:schw}, using the co-area formula and integration by parts. Indeed, for any $u \in C^2(\overline \Omega)$, and any sufficiently small $d_1 >0$, we have that
$$ \int\limits_0^{d_1} \rho'(\epsilon) \| u \|_{H^2(\Omega_{\epsilon}\setminus \overline{\Omega_{d_1}})}^2 \; d\epsilon = \|u\|^2_{H^2_{\rho}(\Omega\setminus \overline{\Omega_{d_1}})}.$$  
With $C_0 = \rho(d_1)=\int\limits_0^{d_1} \rho'(\epsilon)\; d\epsilon>0$ and (\ref{coest-1}), this implies 
\begin{align*}
C_0 (\|u\|^2_{H^2(\Omega_{d_1})}+a^2)+ \|u\|^2_{H^2_{\rho}(\Omega\setminus \overline{\Omega_{d_1}})}  \leq C^2 \int\limits_0^{d_1} \rho'(\epsilon) \|\mathcal S_g^*(u,a)\|^2_{L^2(\Omega_{\epsilon})}\; d\epsilon\\
\leq  C^2 C_0 \|\mathcal S_g^*(u,a)\|^2_{L^2(\Omega_{d_1})} +  C^2 \|\mathcal S_g^*(u,a)\|^2_{L^2_{\rho}(\Omega\setminus \overline{\Omega_{d_1}})}. 
\end{align*}
By the density of $C^2(\overline \Omega)$ in $H^2_\rho (\Omega)$ (cf. \cite[Lemma 2.1]{cs:ak}), (\ref{eq:coest}) now follows easily. \end{proof}


\subsection{Variational solution of the linearized equation} Under the assumption that $\mathcal{S}_g^*$ has trivial kernel, solutions to $\mathcal{S}_g(h)=(\sigma, \tau)$ for $(\sigma, \tau)\in L^{2}_{\rho ^{-1}}(\Omega)\times \mathbb{R}$ can be obtained from a standard variational argument.   

\begin{prop} [\protect{Cf. \cite[Proposition 3.6]{cor:schw}}] \label{varsol} Let $(\overline \Omega, g)$ be as in Theorem \ref{locdef}. Let $(\sigma, \tau)\in L^{2}_{\rho^{-1}}(\Omega)\times \mathbb{R}$. Define the functional $\mathcal{F}:H^{2}_{\rho}(\Omega)\times \mathbb{R}\rightarrow \mathbb R$ by \be \mathcal{F}(u,a)=\int\limits_{\Omega} \left(\frac{1}{2}\mid \mathcal{S}_g^{*}(u,a)\mid^{2}\rho - \sigma u\right)\, d\mu_g -a\tau.\ee  
Then $\mathcal F$ has a unique critical point $(u,a)\in H^2_{\rho}(\Omega)\times \mathbb R$.  This critical point is the global minimizer of $\mathcal F$ and it is a weak solution of the equation 
 $\mathcal S_g (\rho \mathcal S_g^*(u,a))= (\sigma, \tau).$
There is a constant $C>0$ such that for every $(\sigma , \tau)\in L^{2}_{\rho^{-1}}(\Omega)\times \mathbb{R}$ the  minimizer $(u,a) \in H^{2}_{\rho}(\Omega)\times \mathbb{R}$ of the corresponding functional satisfies $\|(u, a)\|_{H^2_{\rho}(\Omega)\times \R}\leq C \|(\sigma, \tau)\|_{L^2_{\rho^{-1}}(\Omega)\times \R}.$
\end{prop}
\begin{proof} Let $\mu=\inf \{\mathcal{F}(u,a): (u,a) \in H^{2}_{\rho}(\Omega)\times \mathbb{R}\}$.  The choice $(u,a)= (0,0)$ shows that $\mu\leq 0$.  The coercivity estimate (\ref{eq:coest}) shows that $\mu$ is finite. Standard Hilbert space arguments exactly as in \cite[p. 150-152]{cor:schw} show that a minimizer $(u,a)\in H^2_{\rho}(\Omega)\times \R$ of $\mathcal F$ exists. 

If $(u, a)\neq (\hat u, \hat a)\in H^2_{\rho}(\Omega)\times \R$, then $\mathcal S_g^*(u-\hat u, a-\hat a)\neq 0$, and the map $t\mapsto\mathcal{F} \left( (1-t)(u,a)+t (\hat u, \hat a) \right)$ is strictly convex. This shows that $(u, a)$ is the unique critical point and in particular the only global minimizer of $\mathcal{F}$.

The Euler-Lagrange condition for the critical point $(u, a)$ of $\mathcal F$ gives that for all $(v,b)\in C^2_c(\Omega)\times \R$, $$\int_{\Omega} \mathcal S_g^*(u,a)\cdot \mathcal S_g^*(v,b)\rho\; d\mu_g = \int_{\Omega} \sigma v \; d\mu_g  + \tau b.$$  Thus $(u, a)$ is a weak solution of  $\mathcal S_g (\rho \mathcal S_g^*(u,a))= (\sigma, \tau)$.

Finally, using the coercivity estimate (\ref{eq:coest}), Cauchy-Schwarz, and $\mu\leq 0$, we obtain that
\begin{eqnarray*}
\frac{1}{2C}\|(u, a)\|_{H^2_{\rho}(\Omega)\times \mathbb{R}}^2&\leq&  \int\limits_{\Omega}\frac{1}{2}\mid \mathcal{S}_g^{*}(u,a)\mid^{2}\rho d\mu_g\\
&=&  \mu+ \int_{\Omega} \sigma u \, d\mu_g +a\tau\\ & \leq & \|(\sigma, \tau)\|_{L^2_{\rho^{-1}}(\Omega)\times \mathbb{R}}\cdot  \|(u,a)\|_{H^2_{\rho}(\Omega)\times \mathbb{R}}. 
\end{eqnarray*}
\end{proof}


\subsection{Pointwise estimates of the variational solution} \label{subsection:pointwise}
We will use the following function spaces: 
\begin{align*}
\mathcal B_0 & := \left( C^{0,\alpha}_{\phi, \phi^{4+\frac{n}{2}}\rho^{-\frac{1}{2}}}(\Omega)\cap L^2_{\rho^{-1}}(\Omega)\right) \times \R\\
\mathcal B_2 &:= C^{2,\alpha}_{\phi, \phi^{2+\frac{n}{2}}\rho^{-\frac{1}{2}}}(\text{Sym}^2(T^*\Omega))\cap L^2_{\rho^{-1}}(\text{Sym}^2(T^*\Omega)) \\
\mathcal B_4&:= \left( C^{4,\alpha}_{\phi, \phi^{\frac{n}{2}}\rho^{\frac{1}{2}}}(\Omega)\cap H^2_{\rho}(\Omega)\right) \times \R
\end{align*} 
with Banach norms
\begin{align*}
\|(\sigma, \tau)\|_0 & := |\tau|+ \|\sigma\|_{L^2_{\rho^{-1}}}+ \|\sigma\|_{C^{0,\alpha}_{\phi, \phi^{4+\frac{n}{2}}\rho^{-\frac{1}{2}}}}\\
\|h\|_2 &: = \|h\|_{L^2_{\rho^{-1}}}+ \|h\|_{C^{2,\alpha}_{\phi, \phi^{2+\frac{n}{2}}\rho^{-\frac{1}{2}}}}\\
\|(u,a)\|_4 & := |a|+\|u\|_{H^2_{\rho}}+ \|u\|_{C^{4,\alpha}_{\phi, \phi^{\frac{n}{2}}\rho^{\frac{1}{2}}}}\; .
\end{align*}
The operator $\rho \mathcal S_g^*$ is continuous from $C^{4,\alpha}_{\phi, \phi^{\frac{n}{2}}\rho^{\frac{1}{2}}}(\Omega)\times \R$ to the space of $C^{2,\alpha}_{\phi, \phi^{2+\frac{n}{2}}\rho^{-\frac{1}{2}}} (\Omega)$ sections of $\text{Sym}^2(T^*\Omega)$. The operator $\mathcal S_g$ is continuous from the space of $C^{2,\alpha}_{\phi, \phi^{2+\frac{n}{2}}\rho^{-\frac{1}{2}}}(\Omega)$ sections of $\text{Sym}^2(T^*\Omega)$ to  $C^{0,\alpha}_{\phi, \phi^{4+\frac{n}{2}}\rho^{-\frac{1}{2}}}(\Omega)\times \R$. 

\begin{prop} \label{prop:ptwise} Let $(\overline \Omega, g)$ be as in Theorem \ref{locdef}. There exists a constant $C>0$ with the following property. Given $(\sigma, \tau)\in \mathcal B_0$, there is $(u,a)\in \mathcal B_4$ so that $\mathcal{S}_g(h)= (\sigma, \tau)$, $\|(u,a)\|_4 \leq C \|(\sigma, \tau)\|_0$, and $\|h\|_2\leq C \|(\sigma, \tau)\|_0$ where  $h= \rho \mathcal S_g^*(u,a)$. 
\end{prop}
\begin{proof} Fix $(\sigma, \tau)\in \mathcal B_0$. Let $(u,a)\in H^2_{\rho}(\Omega)\times \R$ be the weak solution of 
$\mathcal S_g \rho \mathcal S_g^*(u,a) = (\sigma, \tau)$ from Proposition \ref{varsol}. Let $h= \rho \mathcal S_g^*(u,a)\in L^2_{\rho^{-1}}(\Omega)$.  Elliptic regularity for the operator $\rho^{-1}L_g \rho L_g^*$ gives that $h\in C^{2,\alpha}(\Omega)$.  

Note that $\rho^{-1}L_g(\rho L_g^* u)= \rho^{-1} \sigma - \frac{a}{2} \rho^{-1}L_g(\rho g)$. We apply the Schauder interior estimates in the form discussed in Appendix \ref{sec:Schauder}.  We also use the bound $\|(u,a)\|_{H^2_{\rho}\times \R} \leq C \|(\sigma, \tau)\|_0$ from Proposition \ref{varsol} and the obvious estimate $\|h\|_{L^2_{\rho^{-1}}}  \leq C \|(u,a)\|_{H^2_{\rho}\times \R}$. The constant $C$ may change from line to line. 
\begin{eqnarray*} 
\|h\|_{C^{2,\alpha}_{\phi, \phi^{2+\frac{n}{2}}\rho^{-1/2}}} &=&\|\rho \mathcal S_g^*(u, a)\|_{C^{2,\alpha}_{\phi, \phi^{2+\frac{n}{2}}\rho^{-1/2}}} \leq  C \|(u, a)\|_{C^{4,\alpha}_{\phi, \phi^{n/2}\rho^{1/2}}\times \mathbb R}\\
&\leq & C  \Big( \|\rho^{-1} \sigma - \frac{a}{2} \rho^{-1}L_g(\rho g)\|_{C^{0,\alpha}_{\phi, \phi^{4+\frac{n}{2}}\rho^{1/2}}}+ \|(u, a)\|_{L^2_{\rho}\times \mathbb R} \Big) \\
& \leq & C\Big(  \|\rho^{-1} \sigma \|_{C^{0,\alpha}_{\phi,\phi^{4+\frac{n}{2}}\rho^{1/2}}}+ \|(u,a)\|_{L^2_{\rho}\times \mathbb R}\Big)\\
&\leq & C \Big(  \|\sigma \|_{C^{0,\alpha}_{\phi, \phi^{4+\frac{n}{2}} \rho^{-1/2}}}+ \|(\sigma, \tau)\|_{L^2_{\rho^{-1}}\times \mathbb R} \Big)\\
&=& C \|(\sigma, \tau)\|_0.
\end{eqnarray*}
\end{proof}


\subsection{Solving the non-linear problem by iteration} \label{sec:nonlinear}

The goal of this section is to obtain a solution of the non-linear problem $\Theta(g+h)=\Theta(g)+(\sigma, \tau)$  using the linear theory from Section  \ref{subsection:pointwise} to iteratively adjust approximate solutions.  The proof of Theorem \ref{locdef} will be complete once Proposition \ref{non-lin} has been established.  

We first make a general remark about the quadratic remainder term in the Taylor expansion of $h \mapsto \Theta(g + h)$ at an arbitrary $C^{2, \alpha}(\overline \Omega)$ metric $g$. We have that 
$$\Theta (g + h) = (R(g + h), V(g + h)) = (R(g), V(g)) + \mathcal{S}_{g} (h) + Q_g (h)$$ 
where $\mathcal{S}_g = D \Theta_g$ is the linearization of $\Theta$ at $g$ and where $Q_g$ is the ``quadratic remainder" term. More precisely, in a fixed coordinate system, $\mathcal{S}_g (h)$ (respectively $Q_g (h)$) is a homogeneous linear (quadratic) polynomial in $h_{ij}, \partial_k h_{ij}$ and $\partial^2_{k\ell} h_{ij}$ whose coefficients are smooth functions of $g_{ij}, \partial_{k} g_{ij}, \partial_{k\ell}^2 g_{ij}$ (and $h_{ij}, \partial_k h_{ij}$, $\partial_{k\ell}^2 h_{ij}$). It follows that there is a constant $D > 0$ so that for any open subset $U \subset \subset \Omega$ we have that $||Q_g (h)||_{C^{\alpha} (U) \times \mathbb{R}} \leq D || h ||^2_{C^{2, \alpha} (U)}$. Using this estimate for $U=B(x,\phi(x))$ a small ball near the boundary, and for $U$ the complement of a thin collar neighborhood of $\p \Omega$, we obtain 
$$ \|Q_g(h)\|_0 \leq D \|h\|_{2}^2$$
where $D$ might have changed. Here, we also used that the weight $\rho$ tends to zero faster on approach to the boundary than any power of the distance function. Enlarging $D$ slightly if necessary, we also see that 
\begin{eqnarray} \label{eqn:Taylorquadratic} \|Q_\gamma(h)\|_0 \leq D \|h\|_{2}^2 \end{eqnarray}
holds for every metric $\gamma$ that is sufficiently close to $g$ in $C^{2, \alpha}(\overline \Omega)$. Similarly, we have that 
\begin{eqnarray} \label {eqn:Taylorlinear} || \mathcal{S}_{\gamma}(h) - \mathcal{S}_{\gamma'}(h) ||_0 \leq D || h  ||_2 || \gamma - \gamma' ||_2\end{eqnarray} provided that $\gamma, \gamma'$ are $C^{2, \alpha}(\overline \Omega)$ close to $g$.  In (\ref{eqn:Taylorquadratic}) and (\ref{eqn:Taylorlinear}) the weighted $L^2$ and Schauder and norms, in whose definition we use the distance function to the boundary of $\Omega$, are computed with respect to the fixed metric $g$, cf. Remark  \ref{rem:nearbymetrics}.

\begin{prop} \label{non-lin} Let $(\overline \Omega, g)$ be as in Theorem \ref{locdef}. Let $C>0$ be the constant from Proposition \ref{prop:ptwise}.  There exists $\epsilon_0>0$ so that given any $(\sigma, \tau)\in \mathcal B_0$ with $\|(\sigma, \tau)\|_0\leq \epsilon_0$, there exists $(u,a)\in \mathcal B_4$ so that for $h= \rho \mathcal S_g^*(u,a)$, $g+h$ is a metric with $\Theta(g+h)=\Theta(g)+(\sigma, \tau)$, and such that $\|(u,a)\|_4\leq 2 C \|(\sigma, \tau)\|_0$ and $\|h\|_2\leq 2C \|(\sigma, \tau)\|_0$. \end{prop}

\begin{proof}  Let $(u_0, a_0) \in \mathcal{B}_4$ be the solution of $\mathcal{S}_g\rho \mathcal{S}_g^*(u_0,a_0)= (\sigma, \tau)$ from Proposition \ref{prop:ptwise} and let $h_0=\rho \mathcal S_g^*(u_0,a_0)$, so that  $$\|(u_0,a_0)\|_4\leq C\|(\sigma, \tau)\|_0 \text{  and  } \|h_0\|_2\leq C \|(\sigma, \tau)\|_0.$$ From (\ref{eqn:Taylorquadratic}) we obtain that  $\|Q_g(h_0)\|_0 \leq D \|h_0\|_{2}^2$ and hence  
$$\|\Theta(g+h_0)-(R(g)+\sigma, V(g)+\tau)\|_0=\|Q_g(h_0)\|_0\leq D C^2 \|(\sigma, \tau)\|_0^2.$$
We let $\gamma_1:= g + h_0$. Note that $\gamma_1$ is a $C^{2, \alpha}(\overline \Omega)$ metric provided $||(\sigma, \tau)||_0$ is sufficiently small. Fix $\delta \in (0, 1)$. We require that $\epsilon_0 >0$ to be so small that $D C^2 \epsilon_0^{1 - \delta} \leq 1$.  
We now proceed inductively: 

\begin{lemma}  [\protect{Cf. \cite[Proposition 3.9]{cor:schw}}] Fix $\delta\in (0,1)$. Let $C$ be the constant from Proposition \ref{prop:ptwise}.  There exists $\epsilon_0\in (0,\frac{1}{2})$ depending only on $\delta$, $\Omega$, and $g \in C^{4, \alpha}(\overline{\Omega})$ such that the following holds.  Suppose that $m \geq 1$ and that we have constructed $(u_0, a_0), \ldots , (u_{m-1}, a_{m-1}) \in \mathcal{B}_4$, $h_0,\ldots,h_{m-1}\in \mathcal{B}_2$ where $h_p = \rho \mathcal{S}_g^* (u_p, a_p)$, and metrics $\gamma_1,\ldots,\gamma_m \in C^{2,\alpha}(\overline{\Omega})$ where $\gamma_{j} = g + \sum_{p=0}^{j-1} h_p$. Assume that $\|(\sigma, \tau)\|_0\leq \epsilon_0$ and that for all $0 \leq p\leq m-1$,  
\begin{eqnarray} \label{eqn:aux11}  \|(u_p, a_p)\|_4 \leq C \|(\sigma,\tau)\|_0^{(1+p\delta)}     \ \      \text{ and }        \ \       \|h_p\|_2\leq C \|(\sigma,\tau)\|_0^{(1+p\delta)}, \end{eqnarray}
and that for all $1 \leq j\leq m$,
\begin{eqnarray} \label{eqn:aux12} \|\Theta(\gamma_j) - (R(g)+\sigma, V(g)+\tau)  \|_0\leq \|(\sigma,\tau)\|_0^{(1+j\delta)}. \end{eqnarray}
If we define $h_m:=\rho\mathcal S_g^*(u_m,a_m)$ where $(u_m, a_m)$ is the variational solution to $\mathcal S_g \rho \mathcal S_g^*(u_m, a_m)=(R(g)+\sigma, V(g)+\tau)-\Theta(\gamma_m)$ from Proposition \ref{prop:ptwise}, and if we let $\gamma_{m+1}:=\gamma_m+h_m$,  then $\gamma_{m+1}$ is a $C^{2,\alpha}(\overline{\Omega})$ metric and the estimates (\ref{eqn:aux11}) and (\ref{eqn:aux12}) hold for $p=m$ and $j=m+1$.  \label{pic}
\end{lemma}

\begin{proof} We let $\gamma_0 := g$. The induction hypotheses ensure that  $\|g-\gamma_j\|_{C^{2,\alpha}(\overline{\Omega})}$ stays small (depending on $\epsilon_0 > 0$) throughout the iteration. 
Using Proposition \ref{prop:ptwise}, we find $(u_m, a_m)\in \mathcal B_4$ such that $\mathcal S_g \rho \mathcal S_g^*(u_m,a_m)=(R(g)+\sigma, V(g)+\tau)-\Theta(\gamma_m)$. Putting $h_m:=\rho \mathcal S_g^*(u_m, a_m)$, the hypotheses imply the following: 
\begin{align*}
\|(u_m, a_m)\|_4 &\leq C \|(R(g)+\sigma, V(g)+\tau)-\Theta(\gamma_m)\|_0\leq C \|(\sigma, \tau)\|_0^{1+m\delta}, \\ 
\|h_m\|_2 &\leq C \|(R(g)+\sigma, V(g)+\tau)-\Theta(\gamma_m)\|_0\leq C \|(\sigma, \tau)\|_0^{1+m\delta}.
\end{align*}
Note that 
\begin{eqnarray*}
\nonumber \Theta(\gamma_{m+1})&=&\Theta(\gamma_m)+\mathcal{S}_{\gamma_m}(h_m)+Q_{\gamma_m}(h_m)\\&=&(R(g)+\sigma, V(g)+\tau)+\sum\limits_{p=0}^{m-1}[\mathcal{S}_{\gamma_{p+1}}(h_m)-\mathcal{S}_{\gamma_{p}}(h_m)]+Q_{\gamma_m}(h_m).
\end{eqnarray*}
Using (\ref{eqn:Taylorquadratic}), (\ref{eqn:Taylorlinear}) and elementary manipulations, we obtain that
\begin{align*}
& \|(R(g)+\sigma, V(g)+\tau) - \Theta(\gamma_{m+1})\|_0  \leq D \left( \|h_m\|^2_{2}+\|h_m\|_2\sum\limits_{p=0}^{m-1}\|h_p\|_2\right)\\
&\leq D C^2 \left( \|(\sigma, \tau)\|^{(2+2m\delta)}_0 + \|(\sigma, \tau)\|^{2+m\delta}_0\sum\limits_{p=0}^{m-1}\|(\sigma, \tau)\|^{\delta p}_0\right)\\ 
& \leq 2 D C^2 \epsilon_0^{1-\delta}(1 - \epsilon_0^\delta)^{-1} \|(\sigma, \tau)\|_0^{1+(m+1)\delta}.
\end{align*}
Choose $\epsilon_0 >0$ small enough so that $2 D C^2 \epsilon_0^{1-\delta} (1 - \epsilon_0^\delta)^{-1} \leq 1$.
\end{proof}

It follows that the series $\sum_{p=0}^\infty (u_p, a_p)$ converges in $\mathcal {B}_4$ to some $(u, a)$, and that if $h:= \rho \mathcal{S}_g^*(u, a)$, then $\gamma:= g + h$ satisfies $\Theta(\gamma) = (R(g) + \sigma, V(g)+\tau)$. Choosing $\epsilon_0>0$ even smaller if necessary, we obtain that $||(u, a)||_4 \leq 2 C ||(\sigma, \tau)||_0$ and $||h||_2 \leq 2 C ||(\sigma, \tau)||$ from summing the estimates for $(u_p, a_p)$ and $h_p$. This concludes the proof of Proposition \ref{non-lin}. 
\end{proof}

\begin{remark} \label{rem:nearbymetrics} The conclusion of Proposition \ref{non-lin} holds with one choice for $\epsilon_0>0$ and $C>0$ for any metric $g'$ from a small $C^{4, \alpha}(\overline \Omega)$ neighborhood of $g$. To see this, note  that the condition that $\mathcal{S}_{g'}^*$ have only trivial kernel in $H^1_{\mathrm{loc}} (\Omega) \times \mathbb{R}$ is an open condition for $g' \in {C}^{4, \alpha}(\overline \Omega)$. This follows easily from Proposition \ref{prop-bdryexten}. The fundamental coercivity estimate (\ref{coest-1}), and hence (\ref{eq:coest}), holds with a uniform constant $C$ for all metrics $g'$ that are close to $g$ in $C^{4, \alpha} (\overline \Omega)$. The dependence on the metric can easily be made part of the proof.  The derivation of (\ref{coest-1}) is the only indirect argument that was used in the proof of Proposition \ref{non-lin}. We emphasize that the norms of the lower order terms of the operators to which we apply Schauder estimates in the proof of Proposition \ref{prop:ptwise} are uniformly bounded in appropriate spaces, even though the weighted norms as $g'$ varies in a neighborhood of $g$ are not necessarily equivalent.  Thus there is a constant $C$ for which the weighted Schauder estimates will hold for all $g'$ from a $C^{4, \alpha}(\overline \Omega)$ neighborhood of $g$.  
\end{remark}


\subsection{Continuous dependence}

\begin{prop} \label{ctsdep} There exist $\epsilon_0 >0$ and $C>0$ with the following property. If $(\sigma_i, \tau_i) \in \mathcal{B}_0$ with $\|(\sigma_i, \tau_i)\|_0\leq \epsilon_0$ for $i=1, 2$, and if $\gamma_1$ and $\gamma_2$ are the corresponding solutions of $\Theta(\gamma_i)= \Theta(g)+ (\sigma_i, \tau_i)$, $i=1,2$, constructed in the proof of Proposition \ref{non-lin}, then $\|\gamma_1-\gamma_2\|_2\leq C \|(\sigma_1,\tau_1)-(\sigma_2,\tau_2)\|_0$. \end{prop}

\begin{proof} 
Let $(u_i, a_i)$, $h_i=\rho \mathcal S_g^*(u_i, a_i)$, and $\gamma_i= g+h_i$ be as in Proposition \ref{non-lin}. Then
\begin{align*}
\mathcal S_g(h_1-h_2) &= \mathcal S_g\rho \mathcal S_g^*(u_1-u_2, a_1-a_2) = (\sigma_1-\sigma_2, \tau_1-\tau_2)+ (Q_g(h_2)-Q_g(h_1)) .
\end{align*}
Analysis of the remainder term in the Taylor expansion as in Section \ref{sec:nonlinear} gives
\begin{eqnarray*}
\|Q_g(h_1)-Q_g(h_2)\|_0 &\leq& D \| h_1- h_2\|_2 ( \| h_1\|_2+\| h_2\|_2) \leq 2 C D \epsilon_0 \|h_1- h_2\|_2.
\end{eqnarray*}
Interior Schauder estimates for the operator $\rho^{-1} L_g \rho L_g^*$ give that 
\begin{eqnarray} 
\|h_1- h_2\|_2 &=& \|\rho \mathcal S_g^*(u_1-u_2,a_1- a_2)\|_2  \leq  C \|(u_1-u_2,a_1- a_2)\|_{4} \nonumber  \\
& \leq&   C(\|\rho^{-1}(\sigma_1 -\sigma_2) - \rho^{-1}\frac{a_1 - a_2}{2}L_g (\rho g)\|_{C^{0, \alpha}_{\phi, \phi^{4+\frac{n}{2}} \rho^{1/2}}} \nonumber \\  &&+  \nonumber \|(u_1-u_2,a_1- a_2)\|_{H^2_{\rho}\times \mathbb R} + \|Q_g(h_1)-Q_g(h_2)\|_0 ) \nonumber \\
 & \leq & C(\|\sigma_1- \sigma_2\|_0 + \|(u_1-u_2,a_1- a_2)\|_{H^2_{\rho}\times \mathbb R} + 2 C D\epsilon_0 \|h_1- h_2\|_2). \label{estimateaux}
\end{eqnarray}
By the coercivity estimate (\ref{eq:coest}), 
\begin{eqnarray*}
 \|(u_1-u_2,a_1- a_2)\|^2_{H^2_{\rho}\times \mathbb{R}} & \leq & C \int\limits_{\Omega} \mathcal S_g^*(u_1-u_2,a_1- a_2) \cdot \rho \mathcal S_g^*(u_1-u_2,a_1- a_2)\; d\mu_g \\
&= & C \int\limits_{\Omega} \mathcal S_g^*(u_1-u_2,a_1- a_2) \cdot (h_1-h_2)\; d\mu_g .
\end{eqnarray*}
We would like to integrate by parts in the last term.  Since $\mathcal S_g^*(u_1-u_2, a_1-a_2)\in L^2_{\rho}(\Omega)\cap C^{2,\alpha}_{\phi, \phi^{2+\frac{n}{2}}\rho^{\frac{1}{2}}}(\Omega)$, it is not immediately clear that the boundary terms will vanish. Using that $C^{\infty}(\overline{\Omega})$ is dense in $H^2_{\rho}(\Omega)$ (cf. \cite[Lemma 2.1]{cs:ak}), we can justify the integration by parts using an approximation argument. It follows that
\begin{align*} 
&  \|(  u_1-u_2   ,a_1- a_2 )  \|^2_{H^2_{\rho}\times \mathbb{R}}
\leq C \int\limits_{\Omega} (u_1-u_2,a_1- a_2) \cdot \mathcal S_g (h_1-h_2)\; d\mu_g \\
&\leq  C\|(u_1-u_2,a_1- a_2)\|_{L^2_{\rho} \times \R}  \|(\sigma_1- \sigma_2, \tau_1- \tau_2)+(Q_g( h_2)-Q_g( h_1)) \|_{L^2_{\rho^{-1}}\times \mathbb{R}}\\
& \leq C \|(u_1-u_2,a_1- a_2)\|_{H^2_{\rho} \times \R} \left( \|(\sigma_1-\sigma_2,\tau_1-\tau_2)\|_0 +2 C D \epsilon_0 \|h_1- h_2\|_2\right).
\end{align*}
Together with (\ref{estimateaux}), this completes the proof. 
\end{proof}


\subsection{Higher order regularity of the solution and the proof of Theorem \ref{locdefTake2}}

The non-linear differential operator $u \mapsto \hat P (u) = \rho^{-1} \left( R(g + \rho \mathcal{S}_g^* (u, a) ) - R(g)\right)$ is quasi-linear fourth order elliptic in $u$ provided that $\rho \mathcal{S}_g^*(u, a)$ is sufficiently small. The fourth order part of this operator is equal to 
\begin{eqnarray*}
\frac{1}{2}\gamma^{i\ell} \gamma^{jk} g^{a b} \left( - g_{j\ell} u_{abik} + g_{jk} u_{a b i \ell} + g_{i\ell} u_{a b j k} - g_{ik} u_{abj\ell} \right).
\end{eqnarray*}
Here, $\gamma^{ij} := (g + \rho \mathcal{S}_g^* (u, a))^{-1}_{ij}$. To see ellipticity, note that for $\rho \mathcal{S}_g^* (u, a)$ sufficiently small, $\gamma^{ij}$ is close to $g^{ij}$ and the symbol of the operator is close to $(n-1)|\xi|^4$. The lower order terms may blow up on the boundary. The equation $\hat P(u)  = \rho^{-1} \sigma$ can be cast in a form to which higher order Schauder estimates similar to those in (\ref{wtschk}) of Appendix \ref{sec:Schauder} and bootstrapping can be applied. Note that the right hand side is compactly supported and hence lies in any of the weighted Sobolev spaces we defined. Under the regularity assumptions for $(\overline \Omega, g)$ in the statement of Theorem \ref{locdefTake2}, we obtain that
\begin{equation} \label{regest}
\|u\|_{C^{k, \alpha}_{\phi, \phi^{n/2}\rho^{1/2}}}\leq C(k, \Omega, g) \left(\| \hat Pu\|_{C^{k-4, \alpha}_{\phi, \phi^{4+\frac{n}{2}} \rho^{1/2}}} + \|u\|_{L^2_\rho}\right).
\end{equation}
This implies that
\begin{align*}
\|h\|_{C^{k-2,\alpha}_{\phi,\phi^{2+\frac{n}{2}}\rho^{-1/2}}} &= \|\rho \mathcal S_g^*(u, a)\|_{C^{k-2,\alpha}_{\phi,\phi^{2+\frac{n}{2}}\rho^{-1/2}}}\leq C\|(u, a)\|_{C^{k,\alpha}_{\phi,\phi^{n/2}\rho^{1/2}}\times \mathbb R}<\infty.
\end{align*}
In particular, it follows that $h=\rho \mathcal S_g^*(u,a)$ extends by $0$ as a $C^{k-2, \alpha}$ function across the boundary of $\Omega$. We emphasize that we lose two degrees of differentiability in the construction of $h$. Theorem \ref{locdefTake2} follows from this, Remark \ref{rem:nearbymetrics}, and inspection of how the weighted norms we have used are constructed and depend on the metric tensor.  Note that to arrange $\mbox{supp}(\gamma-g)$ to be compactly contained in $\Omega$, we would first replace $\Omega$ by $\Omega_{\delta}$, for $\delta>0$ so small that $\Omega_0\subset \Omega_\delta$ and so that $\Omega_\delta$ is not $V$-static (see the proof of Proposition \ref{prop:coest}).

We remark that if the metric $g$ is smooth to start with, then we can bootstrap to conclude that $h$, and hence $\gamma = g + h$, is also smooth.


\section{Constant scalar curvature gluing with a volume constraint}   \label{sec:gluing}

\begin{thm} \label{thm-imp-strenthed}
Let $k \geq 2$, $n \geq 3$, and $\sigma_n\in \{ -n(n-1), 0, n(n-1)\}$. Let $(\Sigma_1, \gamma_1), (\Sigma_2, \gamma_2)$ be two compact $n$-dimensional $C^{k, \alpha}$ Riemannian manifolds with non-empty boundary. Assume that $R(\gamma_1) = R(\gamma_2)=\sigma_n$.  When $\sigma_n>0$, we also assume that the operators $(\Delta_{\gamma_i}+n)$ have positive Dirichlet spectrum on $ \Sigma_i$. Let $p_i\in \rm{int}(\Sigma_i)$ and $U_i$ be a neighborhood of $p_i$ in $\Sigma_i$ for $i \in \{1, 2\}$. There is a family of $C^{k, \alpha}$ metrics $\{ \hat \gamma_T \} $ on the connected sum $\Sigma_1 \#\Sigma_2\supset (\Sigma_1\setminus U_1)\sqcup (\Sigma_2\setminus U_2)$ with $R(\hat \gamma_T)=\sigma_n$ and such that $\hat \gamma_T \to\gamma_1 \sqcup \gamma_2$ in $C^{k, \alpha}\Big((\Sigma_1\setminus U_1)\sqcup (\Sigma_2\setminus U_2)\Big)$ and  $\rm{vol}(\Sigma_1 \#\Sigma_2,\hat \gamma_T)\rightarrow \rm{vol}(\Sigma_1,\gamma_1)+\rm{vol}(\Sigma_2,\gamma_2)$ as $T \to \infty$. 
\label{imp-vol}
\end{thm}

\begin{remark} 
It is clear from the proof that Theorem \ref{thm-imp-strenthed} also holds in the case where $\partial \Sigma_i=\varnothing$ and $\sigma_n<0$ and in the case where $\sigma_n>0$ and $(\Delta_{\gamma_i}+n)$ has positive spectrum. 
\end{remark}

Theorem \ref{thm-imp-strenthed} augments the result of \cite[Theorem 1.2]{cip}, where an analogous gluing result is formulated for $\sigma_n\leq 0$, without a volume constraint.  We include here the case $\sigma_n>0$, and we also estimate the volume, which we need for our application to Theorem \ref{glue}.  The proof of Theorem \ref{thm-imp-strenthed} follows the approach of the proofs of \cite[Theorem 1]{imp} and \cite[Theorem 3.10]{imp:flds} closely. For completeness and clarity, we repeat many of the arguments rather than simply indicating necessary modifications.  Some of our estimates are slightly sharper than the analogues in \cite{imp, imp:flds}, and there is at least one technical simplification as we do not need to employ weighted H\"{o}lder spaces in our argument. 


\subsection{The approximate solution} \label{subsection:approximate_solution}  Here we construct approximate solutions to the scalar curvature equation on the connected sum.  To do this, we use a conformal rescaling in each of two punctured geodesic balls $B_i\setminus \{ p_i\}\subset U_i$ to produce a metric on each of $\Sigma_i\setminus \{ p_i\}$ with an asymptotically cylindrical end. We identify these ends (after a cut off to an exact cylindrical metric far along the end) to form the connected sum. We then superimpose the two conformal factors used to produce these cylindrical blow ups on the ends to obtain a new conformal factor. This gluing generalizes how the (scalar flat) Schwarzschild metric $(\mathbb{R}^n \setminus \{0\}, (1 + \frac{m}{2|x|^{n-2}})^{\frac{4}{n-2}} \sum_{j=1}^n dx_j^2)$ connects two copies of Euclidean space through a small neck of cross sectional area proportional to $m^{(n-1)/(n-2)}$. Our construction of approximate solutions here follows that of \cite[Section 2]{imp} very closely.     \\

\begin{lemma} [Quasi-polar and quasi-cylindrical coordinates] \label{lem-metric-decay} Let $n, k \geq 2$, let $(M,g)$ be an $n$-dimensional $C^{k, \alpha}$ Riemannian manifold, and let $p \in \mathrm{int} (M)$. There exists $r_0 > 0$ such that for every $\rho \in (0, r_0)$, there exist $C^{k+1, \alpha}$ coordinates $(x^1, \ldots, x^n)$ on an open subset containing $B(p, \rho / 2)$, with $(0, \ldots, 0)$ corresponding to $p$, and such that $g_{ij} = \delta_{ij} + Q_{ij}$, where $Q_{ij}\in C^{k, \alpha}(B(p,\rho/2))$, with $Q_{ij}(0)=0=Q_{ij,\ell}(0)$ for all $i, j, \ell \in \{1, \ldots, n\}$. Let $(r, \theta)$ be polar coordinates in this coordinate system. One can extend $r$ to all of $M\setminus \{p\}$ as a $C^{k, \alpha}$ function with uniform bounds on $r^{\ell-1} |\nabla^{\ell}_gr|_g$ for $\ell=1, \ldots k$, so $r(x)$ agrees with $\mathrm{dist}_{g} (p, x)$ on $B(p, \rho) \setminus B(p, \rho/2)$, such that $\frac{1}{2} r (x) \leq \mathrm{dist}_g(p, x) \leq 2 r(x)$ on all of $B(p, \rho)$, and so that $\frac{r(x)}{\mathrm{dist}_g(p, x)} \to 1$ as $\mathrm{dist}_g(p, x) \to 0$. Changing to cylindrical coordinates $t(x) = - \log r(x)$, we can express the metric $r^{-2} g$ on $B(p, \rho/2)\setminus\{p\}$ in the form $dt^2 + g_{\mathbb S^{n-1}} + e^{-2t}\; \hat h$ where $\hat h\in C^{k, \alpha}_{\mathrm{loc}}(B(p,\rho/2)\setminus\{p\})$. Moreover, $\|\hat h\|_{C^{k,\alpha}([-\log (\rho/2), \infty)\times \mathbb S^{n-1})}< \infty$, where the norm (including covariant derivatives) is taken with respect to the cylindrical metric $dt^2 + g_{\mathbb S^{n-1}}$. 
\end{lemma}
  
 \begin{proof} We can compose any $C^{k+1, \alpha}$ diffeomorphism $\varphi$ of a neighborhood of $p$ in $M$ onto a neighborhood of the origin $\varphi(p)$ in $\R^n$ with a map that is a non-singular linear transformation plus a vector field whose entries are homogenous quadratic polynomials in the coordinates, to obtain a new coordinate system centered at $p$ in which $g_{ij}=\delta_{ij}+Q_{ij}$, with $Q_{ij}(0)=0=Q_{ij,\ell}(0)$ for all $i, j, \ell \in \{1, \ldots, n\}$. Let $\theta: \mathbb{S}^{n-1} \to \R^n$ be the standard embedding of the unit sphere, and let $ \Phi: \R \times \mathbb{S}^{n-1} \to \R^n \setminus \{(0, \ldots, 0)\}$ be the ``cylindrical coordinates map"  $\Phi(t, \theta)=e^{-t} \theta$.  Pulling back $g_{ij}$ by this map, we get
  \begin{align*}
 \Phi^*(g_{ij} dx^i \otimes dx^j)&=e^{-2t} \left(dt^2 + g_{\mathbb{S}^{n-1}} + Q_{ij} (e^{-t} \theta) (- \theta^i dt + d \theta^i) \otimes (- \theta^j d t + d \theta^j)\right)\\ &  =: e^{-2t} (dt^2 + g_{\mathbb{S}^{n-1}} + e^{-2t} \hat h). \end{align*} Using $Q_{ij}(0) = 0 = Q_{ij, \ell} (0)$, the assertions about the decay of $\hat h$ follow readily. 
 \end{proof}

\begin{remark}
In \cite[(9)]{imp}, the weaker decay rate $e^{-2t}(dt^2 + g_{\mathbb{S}^{n-1}} + e^{-t} \hat h)$ with $\hat h$ and its derivatives bounded as $t \to \infty$ is used. Our sharper estimate leads to better bounds in some places than those obtained in \cite{imp}.  
\end{remark}  

\begin{remark}
We do not work with cylindrical coordinates based on geodesic polar coordinates in Lemma \ref{lem-metric-decay} to avoid an unnecessary loss of regularity. Recall that the distance function of a $C^{k, \alpha}$ metric to a point $p$ \emph{is} $C^{k, \alpha}$ in a punctured neighborhood of $p$, see \cite{Foote}. Note that one could arrange the coordinates $(x^1, \ldots, x^n)$ to be smooth, say, by starting with a smooth diffeomorphism $\varphi$ in the above proof.
\end{remark}
 
We now fix $R \in (0, \frac{1}{3} \min\{1, \textrm{dist}_{\gamma_i}(p_i, \partial\Sigma_i), (r_{0})_i, i=1,2\})$. Here, $(r_{0})_i$ is as in Lemma \ref{lem-metric-decay} applied with $M = \Sigma_i$ and $p = p_i$. We let $r_{(i)}(x)$ be the functions constructed in Lemma \ref{lem-metric-decay} for $\rho= R$. 

We define $r_i(x)= \min\{ 2R, r_{(i)}(x)\}$. Let $\psi:(0,\infty)\rightarrow (0,\infty)$ be a smooth, positive function with 
$$\psi(t)= \begin{cases} t^{\frac{n-2}{2}} \qquad \qquad \qquad 0< t < R \\ \textrm{interpolation } \quad \quad R\leq t \leq 2R \\ 1 \qquad \qquad \qquad\quad\;\; t > 2R\; .\end{cases}$$ 
We let $\Psi_i(x)= \psi(r_i(x))$. On $B_{\gamma_i}(p_i, R)$, $(\Psi_i(x))^{\frac{4}{n-2}}= (r_i(x))^2$.  Let $\Sigma_i^*=\Sigma_i \setminus \{ p_i\}$. Then $(\Sigma_i^*, \Psi_i^{-\frac{4}{n-2}} \gamma_i)$ is $(\Sigma_i\setminus B_{\gamma_i}(p_i, 2R), \gamma_i)$ with an infinite, asymptotically cylindrical end attached.  

Let $T\geq T_0 \gg -2 \log R >1$, and let $(r_i, \theta)$ be the quasi-geodesic polar coordinates on $B_{\gamma_i}(p_i, 3 R) \subset \Sigma_i$ of Lemma \ref{lem-metric-decay}.  Let $s_i= -\log r_i + \log R -\frac{T}{2}$. Note that under this change of variables, $r_i=R$ corresponds to $s_i=-\frac{T}{2}$ and $r_i\searrow 0$ corresponds to $s_i \nearrow \infty$.  By Lemma  \ref{lem-metric-decay}, the metric $\Psi_i^{-4/(n-2)} \gamma_i$ on $B_{\gamma_i}(p_i, R)$ can be written as $ds_i^2 + g_{\mathbb S^{n-1}}+ e^{-T} e^{-2s_i} R^2 \hat h_i$ where $\hat h_i$ and its covariant derivatives with respect to the cylindrical metric $ds_i^2 + g_{\mathbb{S}^{n-1}}$ are bounded,  independently of $T$.  Let $\gamma_{i, T}$ be the metric obtained by transitioning smoothly in the $(s_i, \theta)$-region $(-1, -\frac{1}{2})\times \mathbb S^{n-1}$ from the metric $\Psi_i^{-4/(n-2)} \gamma_i=ds_i^2 + g_{\mathbb S^{n-1}}+ e^{-T} e^{-2s_i} R^2 \hat h_i$ on $(-\frac{T}{2}, -1)\times \mathbb S^{n-1}$ to the exact cylindrical metric $ds_i^2 +g_{\mathbb S^{n-1}}$ on $(-\frac{1}{2}, \infty) \times \mathbb S^{n-1}$. Such a transition can be accomplished using a cut-off function whose norms do not depend on $T$.   

The cylindrical ends of the two Riemannian manifolds so obtained can be identified by forming the quotient via the relation $(s_1, \theta)\sim (-s_2, \theta)$ in the exactly cylindrical pieces $(-\frac{1}{2}, \frac{1}{2}) \times \mathbb S^{n-1}$ in each of $(\Sigma_i^*, \gamma_{i,T})$. We obtain a new manifold $(\Sigma_T, \gamma_T)$, where $\Sigma_T$ is topologically just $\Sigma_1 \#\Sigma_2$.  We define a new linear coordinate  $s$ on $[-\frac{T}{2}, \frac{T}{2}]\times \mathbb S^{n-1}\cong C_T\subset \Sigma_T$, so that $s=s_1$ for $s\leq 0$ and $s=-s_2$ for $s\geq 0$.  On $C_T$, the metric $ \gamma_T $ takes the form
\begin{equation} \label{metric-est} 
\gamma_T = ds^2+ g_{\mathbb S^{n-1}}+ e^{-T} \cosh (2s) \hat h_T, 
\end{equation} 
where $\hat h_T$ and its derivatives with respect to the cylindrical metric are bounded independently of $T$, and where $\hat h_T = 0 $ on $[ - \frac12, \frac12] \times \mathbb{S}^{n-1}$.

Let $\chi_{1,T}(s)$ be a cut-off function transitioning smoothly from 1 near $\{\frac{T}{2}-1\}\times \mathbb S^{n-1}$ to 0 near $\{ \frac{T}{2}\} \times \mathbb S^{n-1}$, and let $\chi_{2,T}(s)$ be the corresponding cut-off function transitioning smoothly from 1 near $\{-\frac{T}{2}+1\}\times \mathbb S^{n-1}$ to 0 near $\{ -\frac{T}{2}\} \times \mathbb S^{n-1}$. The function $\Psi_T$ defined by $\Psi_T(s, \theta):=\chi_{1,T}(s)\psi (R e^{-s- \frac{T}{2}})+ \chi_{2,T}(s) \psi( R e^{s - \frac{T}{2}})$ extends smoothly to $\Sigma_T$.    

Note that on $(-\frac{T}{2}+1, \frac{T}{2}-1)\times \mathbb S^{n-1}\subset C_T$, we have 
\begin{align*}
\Psi_T(s,\theta)=  (Re^{-s-\frac{T}{2}})^{\frac{n-2}{2}}+  (Re^{s-\frac{T}{2}})^{\frac{n-2}{2}} = 2 R^{\frac{n-2}{2}} e^{-\frac{(n-2)T}{4}} \cosh \big( \tfrac{(n-2)s}{2}\big).
\end{align*}
The scalar curvature of $\Psi_T^{{4/(n-2)}}\gamma_T$ is given by 
\begin{align}
R(\Psi_T^{\frac{4}{n-2}} \gamma_T)=c_n^{-1} \Psi_T^{-\frac{n+2}{n-2}} \left( - \Delta_{\gamma_T} \Psi_T + c_n R(\gamma_T) \Psi _T\right). \label{conscal} 
\end{align}
Here, $c_n = \frac{n-2}{4(n-1)}$. Because $\gamma_T$ is exactly cylindrical  on $(-\frac{1}{2} , \frac{1}{2}) \times \mathbb S^{n-1}$, we have that $R(\Psi_T^{{4/(n-2)}} \gamma_T)=0$ there.  In fact, on this region, $\Psi_T^{{4/(n-2)}}\gamma_T$ is precisely the \emph{Schwarzschild metric} of mass $m= 2 R^{(n-2)}e^{-(n-2)T/2}$, with the minimal sphere at $s=0$ in our coordinates (cf. \cite{bray-lee}).  The geometry of $(\Sigma_T, \Psi_T^{{4/(n-2)}} \gamma_T)$ is thus that of a Schwarzschild neck together with interpolating regions joining $(\Sigma_1\setminus B_{\gamma_1}(p_1, 2R), \gamma_1)$ and $(\Sigma_2\setminus B_{\gamma_2}(p_2, 2R), \gamma_2)$.

We use $\{ (\Sigma_T,   \Psi_T^{{4/(n-2)}} \gamma_T)\}_{T \geq T_0}$ as a family of approximate solutions to the constant scalar curvature equation on $\Sigma_1 \#\Sigma_2$ that we want to perturb to obtain exact solutions. The injectivity radius of $(\Sigma_T, \gamma_T)$ is bounded below uniformly as $T \to \infty$. Define the operator 
$$\mathcal N_T(f):= -\Delta_{\gamma_T} f + c_n R(\gamma_T) f - c_n \sigma_n f^{\frac{n+2}{n-2}}.$$
In view of (\ref{conscal}), we would like to solve $\mathcal N_T(\Psi_T + \eta_T )=0$ for small $\eta_T $, such that $ \Psi_T + \eta_T > 0$. To accomplish this by perturbation, we will estimate $\mathcal N_T(\Psi_T)$ and analyze the mapping properties of the linearization $\mathcal L_T$ of $\mathcal N_T$ at $\Psi_T$.

Let $\|f\|_{k,\alpha}:=\|f\|_{C^{k, \alpha}(\Sigma_T)}$ denote the H\"{o}lder norm on $(\Sigma_T, \gamma_T)$.  Let ``$\lesssim$" denote an inequality up to multiplication by a constant that is independent of $T$. We begin by estimating $\mathcal N_T(\Psi_T)$: 

\begin{prop}  [\protect{Cf. \cite[Proposition 6]{imp}}]  \label{prop-NT}
Let $k\geq 2$. If $3 \le n \le 6$, we have that \begin{equation} \label{approxsol1} \|\mathcal N_T(\Psi_T)\|_{k-2, \alpha} \lesssim e^{-\frac{(n-2)T}{2}}.\end{equation} 
  For $n > 6$, 
  \begin{equation} \label{approxsol2} \|\mathcal N_T(\Psi_T)\|_{k-2,\alpha}\lesssim  e^{-\frac{(n+2)T}{4} }. \end{equation}
\end{prop}

This estimate follows along the lines of  \cite[Proposition 6]{imp},  \cite[Proposition 3.6]{imp:flds}. For the reader's convenience, we include a proof in Appendix \ref{sec:NT}.  The cited proofs involve some additional terms coming from the Einstein constraint equations.  Our estimates are also slightly sharper due in part to the fact that we use a better estimate on the metric (Lemma \ref{lem-metric-decay}).  


\subsection{The linearized operator} 
The linearization $\mathcal L_T$ of $\mathcal N_T$ at $\Psi_T$ is given by  
$$\mathcal L_T(f):= D\mathcal N_T\big|_{\Psi_T} (f)= -\Delta_{\gamma_T} f + c_n R(\gamma_T) f - c_n \sigma_n \frac{n+2}{n-2} \Psi_T^{{4/(n-2)}} f .$$
For a given integer $\ell$ with $0 \leq \ell \leq k$, we define the function space $\mathring C^{\ell,\alpha} (\Sigma_T)=\{ u  \in C^{\ell, \alpha} (\Sigma_T) 
\ \mathrm{and} \  u\big|_{\partial \Sigma_T}=0\}$.  
The proof of the following proposition is very similar to \cite{imp, imp:flds}. We include the proof for completeness and clarity. We let $\Sigma_{i,r}^*= \Sigma_i \setminus B_{\gamma_i } (p_i, r)$.  

\begin{prop} [\protect{Cf. \cite[Proposition 8]{imp}}]  \label{prop-LT}
Let $k \geq 2$. For all sufficiently large $T$, the operators $\mathcal L_T:\mathring C^{k, \alpha} (\Sigma_T)\rightarrow C^{k-2,\alpha} (\Sigma_T)$ are invertible. The norms of the inverse operators $\mathcal G_T: C^{k-2,\alpha} (\Sigma_T)\rightarrow \mathring C^{k,\alpha} (\Sigma_T)$ are uniformly bounded. \end{prop}

\begin{proof} 
We first show that  $\mathcal L_T: \mathring C^{2,\alpha}(\Sigma_T) \rightarrow C^{0, \alpha}(\Sigma_T)$ is invertible for large $T$. The invertibility of these operators for $k\geq3$ follows from elliptic regularity. This map is Fredholm of index zero, so we only have to show that it is injective for $T$ large enough.
We do this by contradiction below.  First, note that $s=(-1)^j$ corresponds to $r_j= Re^{1- \frac{T}{2}}=:r(T)$, and that $(\Sigma^*_{j,r(T)}, \Psi_j^{-{4/(n-2)}}\gamma_j)\subset (\Sigma_T, \gamma_T)$.  Thus, as $T$ grows, more and more of $(\Sigma_j^*,  \Psi_j^{-{4/(n-2)}}\gamma_j)$ is contained in $(\Sigma_T, \gamma_T)$.
 
Suppose there is a sequence $T_m\nearrow \infty$ and non-zero $\eta_m \in \mathring C^{2, \alpha} (\Sigma_{T_m})$ so that $\mathcal L_{T_m}(\eta_m)=0$.  By normalization, we may arrange $\max\limits_{\Sigma_{T_m}} |\eta_m| = 1$. We distinguish two cases. 

In the first case, we assume that for one of $j=1, 2$, and for some $0<r<R$, there is a $c>0$ so that $\max\limits_{\Sigma_{j,r}^*} |\eta_m|\geq c$ (at least for a subsequence, which we re-index).  Let $\tilde{\gamma}_j:=\Psi_j^{-{4/(n-2)}} \gamma_j$.  The operators $\mathcal L_{T_m}$ converge locally on $ \Sigma_j^*$ to the operator $\mathcal L_j = -\Delta_{\tilde \gamma_j} + c_n R(\tilde \gamma_j)- c_n \sigma_n \frac{n+2}{n-2} \Psi_j^{{4/(n-2)}}$.  
Since the $\eta_m$ are uniformly bounded, interior Schauder estimates on the equations $\mathcal L_{T_m} (\eta_m)=0$ imply that we can take a subsequence converging in $C^2$ on compact subsets of $\Sigma_j^*$, to a non-trivial limit function $\eta$ on $\Sigma_j^*$.  Moreover, we have $\mathcal L_j(\eta)=0$ on $\Sigma_j^*$.  Applying the identity 
$$-\Delta_{\tilde\gamma_j} f + c_n R(\tilde \gamma_j) f = \Psi_j^{\frac{n+2}{n-2}}(-\Delta_{\gamma_j} (\Psi_j^{-1}f) + c_n R(\gamma_j) (\Psi_j^{-1}f))$$
for the conformal Laplacian (valid for every $f \in C^2(\Sigma_j)$) with $f = \eta$, we obtain that 
\begin{align*}
c_n \sigma_n \frac{n+2}{n-2} \Psi_j^{\frac{4}{n-2} } \eta = \Psi_j^{\frac{n+2}{n-2}}(-\Delta_{\gamma_j} (\Psi_j^{-1}\eta) + c_n R(\gamma_j) (\Psi_j^{-1}\eta)).
\end{align*}
Since $R(\gamma_j)=\sigma_n$, we conclude that
$$0=\Delta_{\gamma_j} (\Psi_j^{-1}\eta) + \frac{\sigma_n}{n-1} (\Psi_j^{-1}\eta).$$ 

Note that $0 \not \equiv \Psi_j^{-1} \eta \in C^{2, \alpha} (\Sigma_j^*)$ vanishes on $\partial \Sigma_j$. We have that $|\Psi_j^{-1}\eta|\lesssim r_j^{-\frac{n-2}{2}}$ near $p_j$ since $|\eta| \leq 1$. Our assumption that the operator  $\Delta_{\gamma_j} + \frac{\sigma_n}{n-1}$ has positive Dirichlet spectrum on $\Sigma_j$ implies the existence of a positive (Dirichlet) Green's function with pole at $p_j$, where it grows like $r_j^{- (n-2)}$. A standard application of the maximum principle shows that $|\Psi_j^{-1} \eta|$ lies below \emph{any} positive multiple of this Green's function. Hence $\Psi_j^{-1} \eta \equiv 0$, a contradiction. 

If we are not in the first case, then $\eta_m \to 0$ locally uniformly on $\Sigma_j^*$. If $q_m$ is such that $|\eta_m(q_m)|= 1$, then $q_m=(s(q_m), \theta(q_m)) \in C_{T_m} \cong [- \frac{T_m}{2}, \frac{T_m}{2}] \times \mathbb{S}^{n-1}$ and $\min \{ \frac{T_m}{2} - s(q_m), s(q_m) + \frac{T_m}{2}\} \to + \infty$.  Introducing a new linear variable $\tilde s = s - s (q_m)$ (where we are now identifying the cylindrical pieces $C_{T_m} \subset \Sigma_{T_m}$ by identifying the $(s, \theta)$ coordinate patches), we conclude that $\gamma_{T_m} \to \mathring \gamma = d \tilde s^2 + g_{\mathbb{S}^{n-1}}$ and $\Psi_{T_m}\rightarrow 0$ locally smoothly on $\mathbb{R} \times \mathbb{S}^{n-1}$.  It follows that $\mathcal{L}_{T_m} \to \mathring {\mathcal L} := - \Delta_{\mathring \gamma} + c_n R(\mathring{\gamma}) =  - \Delta_{\mathring \gamma} + \frac{(n-2)^2}{4}$ locally smoothly on $\mathbb{R} \times \mathbb{S}^{n-1}$. Using interior Schauder estimates and the supremum bound on $\eta_m$, we get a subsequence converging in $C^2$ on compact subsets of the cylinder to a solution $\eta$ of $\mathring{\mathcal L} (\eta)=0$. Moreover, $|\eta| \le 1$ and $| \eta | = 1$ at some point with $\tilde s =0$. This contradicts the maximum principle. 

Finally, we show that  the norm of the inverse $\mathcal G_T: C^{k-2, \alpha} (\Sigma_{T}) \rightarrow \mathring C^{k, \alpha} (\Sigma_{T}) $ can be bounded independently of $T$ large. The proof proceeds by contradiction, as above. Suppose there are $T_m \nearrow \infty$ and $\eta_m\in C^{k-2, \alpha} (\Sigma_{T_m})$ so that $\|\eta_m\|_{k-2, \alpha}\rightarrow 0$ while $\|\mathcal G_{T_m}(\eta_m)\|_{k, \alpha}=1$.  Let $v_m=\mathcal G_{T_m}(\eta_m)$, so that $\| \mathcal L_{T_m}(v_m)\|_{k-2,\alpha}=\|\eta_m\|_{k-2,\alpha} \rightarrow 0$.  Since $\| v_m \|_{k, \alpha} = 1 $, we see that $v_m$ converges in $ C^{k}$ on compact subsets of $\Sigma^*:=\Sigma_1^*\cup \Sigma_2^*$.  
Just as above, there are two possible cases.  In the first case, for either $j=1$ or $j=2$, some subsequence of the $v_m$ converges in the $ C^{k}$ on compact subsets of $\Sigma_j^*$ to a non-trivial solution $v$ of the equation $\mathcal L_j(v)=0$ with $ v |_{\p \Sigma_j} = 0$, in which case 
$\Psi_j^{-1}v$ extends to a non-trivial element in the kernel of $(\Delta_{\gamma_j} + \frac{\sigma_n}{n-1})$ on $\Sigma_j$. This is a contradiction.  

In the second case, $v_m$ converges to zero in $C^{k}$ on any compact subset of $\Sigma_1^*\cup \Sigma_2^*$.  The operators $ \mathcal L_T $ are uniformly elliptic. Since $\|  \mathcal L_{T_m} (v_m)\|_{k-2,\alpha}\rightarrow 0$ and $\| v_m\|_{k, \alpha}=1$, interior Schauder estimates imply that $\| v_m \|_{0}$ cannot tend to zero.  Thus there is a $c>0$ so that $$c\leq \max\limits_{C_{T_m}} | v_m | \leq 1.$$  The same rescaling to a cylinder as above leads to a contradiction. 
\end{proof}

We have now established the linear theory.  Before moving on to the non-linear estimates, we note that in the zero scalar curvature case $ \sigma_n = 0 $, the problem we want to solve is linear: $\mathcal N_T(f)= -\Delta_{\gamma_T} (f) + c_n R(\gamma_T) f  = \mathcal L_T (f)$. In this case,  Proposition \ref{prop-LT} is enough to obtain a solution $ \eta_T \in \mathring{C}^{k, \alpha}(\Sigma_T) $ of $ \mathcal N_T (\Psi_T + \eta_T ) = 0 $,   
for sufficiently large $ T$.
Such an $ \eta_T$ is  given by
$  \eta_T = - \mathcal G_T ( \mathcal N_T (\Psi_T )) $
and satisfies $ \| \eta_T \|_{k, \alpha} \lesssim \| \mathcal{N}_T (\Psi_T ) \|_{k-2, \alpha} $. 
By Proposition  \ref{prop-NT}, we have that
\begin{align}
\sup_{\Sigma_T} \Big|   \frac{\eta_T}{ \Psi_T}    \Big|
\lesssim \begin{cases} 
e^{-\frac{(n-2)T}{4}}   \qquad \textrm{for }  n \leq 6 \\
e^{-T}   \qquad \qquad \textrm{for }  n > 6.
\end{cases} 
 \end{align}
This guarantees that 
 $$ \Psi_T + \eta_T = \Psi_T \left( 1 + \frac{\eta_T}{ \Psi_T}     \right) > 0 $$
 on $ \Sigma_T $ for sufficiently large $T$.


\subsection{Non-linear estimates}  When $ \sigma_n \neq 0$, we  solve the non-linear problem $\mathcal N_T( \Psi_T + \eta_T) =0$ using a contraction mapping argument. To do this, we apply the linear estimates above, along with the following estimate of the \emph{quadratic error term} $\mathcal Q_T$, where 
\begin{align*}
\mathcal Q_T(\eta) : & = \mathcal N_T(\Psi_T+\eta)-(\mathcal N_T(\Psi_T)+ \mathcal L_T(\eta))\\
&=  c_n \sigma_n \left( \Psi_T^{\frac{n+2}{n-2}} + \frac{n+2}{n-2} \Psi_T^{{4/(n-2)}} \eta - (\Psi_T+\eta)^{\frac{n+2}{n-2}}\right) . 
\end{align*}

The arguments in this subsection follow those of \cite[Section 6]{imp} closely. 

\begin{lemma}  [\protect{Cf. \cite[Lemma 8]{imp}}]  \label{lem-quadratic} Let $k \geq 2$. For all $\eta_i$ with $| \Psi_T^{-1}\eta_i|  \leq \frac{1}{4}$ we have that
\begin{align} \label{quad-est}
\|\mathcal Q_T(\eta_1)-\mathcal Q_T(\eta_2)\|_{k-2, \alpha}\lesssim \left(  \max\limits_{i=1, 2} \| \Psi_T^{-1} \eta_i\|_{k-2,\alpha} \right) \|\eta_1-\eta_2\|_{k-2, \alpha}.
\end{align} 
\end{lemma}
\begin{proof}  This follows at once from the expansion
\begin{align*}
\mathcal Q_T(\eta_1) &-  \mathcal Q_T(\eta_2) = c_n \sigma_n \left[ \frac{n+2}{n-2} \Psi_T^{\frac{4}{n-2}}(\eta_1-\eta_2)+ (\Psi_T+\eta_2)^{\frac{n+2}{n-2}}- (\Psi_T+\eta_1)^{\frac{n+2}{n-2}}\right]\\
&= c_n \sigma_n \frac{n+2}{n-2}(\eta_1-\eta_2) \int\limits_0^1 \left[ \Psi_T^{\frac{4}{n-2}}- \left( \Psi_T + t \eta_1 + (1-t)\eta_2\right)^{\frac{4}{n-2}}\right] dt \\
&= c_n \sigma_n \frac{n+2}{n-2}(\eta_1-\eta_2) \Psi_T^{\frac{4}{n-2}} \int\limits_0^1 \left[ 1- \left( 1 + t \Psi_T^{-1}\eta_1 + (1-t)\Psi_T^{-1}\eta_2\right)^{\frac{4}{n-2}}\right] dt . 
\end{align*}
\end{proof}

\begin{prop}   [\protect{Cf. \cite[Proposition 9]{imp}}] \label{prop:fixed_point}  Let $k \geq 2$. Let $\beta$ be a constant such that $\beta \in (\frac{n-2}{4}, \frac{n-2}{2})$ when $3\leq n \leq 6$, and $\beta \in (\frac{n-2}{4}, \frac{n+2}{4})$ when $n > 6$. 
Let $\mathbb B_{\beta}:=\{ \eta \in \mathring C^{k, \alpha} (\Sigma_T): \|\eta\|_{k, \alpha} \leq e^{-\beta T}\}$. For sufficiently large $T$, the mapping $$F_T:\eta \mapsto -\mathcal G_T (\mathcal N_T(\Psi_T)+\mathcal Q_T(\eta))$$ is a contraction mapping $F_T:\mathbb B_{\beta} \rightarrow \mathbb B_{\beta}$.  
\end{prop}
\begin{proof}  Recall that ``$\lesssim$" denotes an inequality up to multiplication by a constant that is independent of $T$. In view of the explicit expression of $\Psi_T$ in Section \ref{subsection:approximate_solution}, we easily see that for $ \eta \in \mathbb B_{\beta}$ we have that
\begin{align} \label{contract-est} \|\Psi_T^{-1}\eta\|_{k-2,\alpha} \lesssim e^{-\beta T} e^{\frac{(n-2)}{4}T}.
\end{align}  
Since $\beta>\frac{n-2}{4}$, the right-hand side of \eqref{contract-est} goes to $0$ uniformly as $T\rightarrow \infty$. Therefore, for any $\eta_1, \eta_2\in \mathbb B_{\beta}$ and sufficiently large $T$, using the uniform bound on $\mathcal G_T$ from Proposition \ref{prop-LT}, we have
\begin{align}
\|F_T(\eta_1)-F_T(\eta_2)\|_{k, \alpha} &= \|\mathcal G_T (\mathcal Q_T (\eta_2)-\mathcal Q_T(\eta_1) ) \|_{k, \alpha} \nonumber \\ 
& \lesssim \|\mathcal Q_T(\eta_2)-\mathcal Q_T(\eta_1)\|_{k-2, \alpha}.  \nonumber 
\end{align} 
By (\ref{quad-est}) and (\ref{contract-est}), $F_T$ is a contraction mapping on $\mathbb B_{\beta}$ for $T$ large.  To see that $F_T$ maps $\mathbb B_{\beta}$ into itself, we note that  by Proposition \ref{prop-NT} and the upper bound for $ \beta$ we have that $\|\mathcal N_T (\Psi_T)\|_{k-2, \alpha}= o(e^{-\beta T})$, while (\ref{quad-est}) and (\ref{contract-est}) imply that $\|\mathcal Q_T(\eta)\|_{k-2,\alpha}= o(e^{-\beta T})$.  Using once more the $T$-independent bound for the norm of $\mathcal G_T$ from Proposition \ref{prop-LT},  we conclude that indeed $F_T(\eta)\in \mathbb B_{\beta}$ for $\eta \in \mathbb B_{\beta}$.
\end{proof}

Choose $\beta$ as in Proposition \ref{prop:fixed_point}.  For sufficiently large $T$, $F_T$ has a unique fixed point $\eta_T\in \mathbb B_{\beta}$.  If we let $\hat \Psi_T := \Psi_T + \eta_T$, then $\mathcal N_T(\hat \Psi_T)=0$.  Since $\eta_T\in \mathbb B_{\beta}$, by \eqref{contract-est}, we have that $\hat \Psi_T>0$ for large $T$.  Thus we have solved the constant scalar curvature equation $R(\hat\Psi_T^{{4/(n-2)}} \gamma_T)=\sigma_n$.  Moreover, by elliptic regularity, $\hat \Psi_T \in C^{k, \alpha}(\Sigma_T)$.  Let $\hat \gamma_T=\hat\Psi_T^{{4/(n-2)}} \gamma_T$.  

We remark that if the metrics $g_i$ are smooth to start with, then we can bootstrap to conclude that $\hat \gamma_T$ is also smooth. 


\subsection{Volume estimate}  
We now derive estimates on the volume of $(\Sigma_T, \hat \gamma_T)$.  
The following proposition will complete the proof of Theorem \ref{imp-vol}. 
\begin{prop} The volume of $(\Sigma_T, \hat \gamma_T)$ approaches $\mathrm{vol}(\Sigma_1, \gamma_1)+\mathrm{vol}(\Sigma_2,  \gamma_2)$ as $T\rightarrow \infty$. 
\end{prop}

\begin{proof} We note that $\Sigma_T \setminus ([-\frac{T}{4}, \frac{T}{4}]\times \mathbb S^{n-1})$ corresponds to $[ \Sigma_1\setminus B_{\gamma_1}(p_1, e^{-\frac{T}{4}} R)] \cup [\Sigma_2 \setminus B_{\gamma_2} (p_2, e^{-\frac{T}{4}}R)]$.  On the respective components, we have $\Psi_i^{{4/(n-2)}}\gamma_T = \gamma_i$.  Furthermore, $\frac{ \Psi_T}{\Psi_i} \rightarrow 1$ uniformly as $T\rightarrow \infty$ on each of the respective components, and by (\ref{contract-est}), $\frac{\hat \Psi_T}{\Psi_T} \rightarrow 1$ uniformly as well.  Thus the volume of $(\Sigma_T \setminus ([-\frac{T}{4}, \frac{T}{4}]\times \mathbb S^{n-1}), \hat \gamma_T)$ tends to $\mbox{vol}(\Sigma_1, \gamma_1)+\mbox{vol}(\Sigma_2,  \gamma_2)$.   

We show now that $\mbox{vol}([-\frac{T}{4}, \frac{T}{4}]\times \mathbb S^{n-1}, \hat \gamma_T)$ tends to zero as $T \to \infty$. Since  $\mbox{vol} ([-1, 1]\times \mathbb S^{n-1}, \gamma_T)$ is uniformly bounded,  by the estimate of $\Psi_T$, we see that $\mbox{vol} ([-1, 1]\times \mathbb S^{n-1}, \Psi_T^{{4/(n-2)}} \gamma_T)$  goes to zero.  By (\ref{contract-est}), we have $\left| \frac{\hat \Psi_T}{\Psi_T}\right| \leq 2$ on $\Sigma_T$ for all $T$ sufficiently large. Thus $\mbox{vol} ([-1, 1]\times \mathbb S^{n-1}, \hat \gamma_T)$ also goes tends to zero. 
 
We next consider the right half $[1, \frac{T}{4}]\times \mathbb S^{n-1}$. Now, $\mbox{vol}([1, \frac{T}{4}]\times \mathbb S^{n-1}, \Psi_2^{{4/(n-2)}} \gamma_T)$ is less than $\mbox{vol}(B_{\gamma_2}(p_2, e^{-\frac{T}{4}}R), \gamma_2)$ and hence tends to $0$ as $T \to \infty$. In this region, we have that  $\left| \frac{\Psi_T}{\Psi_2}\right| \leq 2$ and $\left| \frac{\hat \Psi_T}{\Psi_T}\right| \leq 2$. Hence $\mbox{vol}([1, \frac{T}{4}]\times \mathbb S^{n-1}, \Psi_2^{{4/(n-2)}} \hat \gamma_T)$ tends to $0$. The left half is dealt with similarly.   
\end{proof} 


\section{Localizing the gluing: Proof of Theorem \ref{glue}}

The idea of combining the theory of local scalar curvature deformation in conjunction with a conformal gluing method in the proof here is exactly as in \cite[p. 57-58]{cd}.

\begin{proof} Fix two points $p_i \in U_i$. There exists $\rho_0 >0$ such that $B_{g_i} (p_i, \rho_0) \subset U_i$, such that for any $\rho \in (0, \rho_0)$ the operators $\Delta_{g_i} + \frac{ \sigma_n}{n-1}$ have positive Dirichlet spectrum on $B_{g_i} (p_i, \rho)$, and such that $U_i \setminus \overline{B_{g_i} (p_i, \rho/2)}$ is not $V$-static. (The last assertion follows from an argument as in the proof of Proposition \ref{prop:coest}.) Fix $\rho \in (0, \rho_0)$. 

Applying Theorem \ref{imp-vol} with $\Sigma_i:= \overline{ B_{g_i} (p_i, \rho)}$, $\gamma_i:=g_i$, we get a family of metrics $\{ \hat g_T: = \gamma_T \}$, with $R(\hat g_T ) = \sigma_n $ on $\Sigma_1 \#\Sigma_2\supset \Sigma_i \setminus B_{g_i} (p_i, \rho/2)$, such that $ \hat g_T $ converges to $ g_i $ in $C^{k, \alpha} (\Sigma_i \setminus B_{g_i} (p_i, \rho/2) )$ and $\mbox{vol}(\Sigma_1 \#\Sigma_2, \hat g_T) \rightarrow  \mbox{vol} (\Sigma_1, g_1)+ \mbox{vol} (\Sigma_2, g_2)$.  

We patch back in the original metric $g_i$, transitioning from $g_i$ near $\partial \Sigma_i$ to $\hat g_T$ near $\partial B_{g_i}(p_i, \rho/2)$ in the usual way: 
let $  0 \le \chi_i  \le 1 $ be a fixed smooth function on  $M_i$ 
such that
 $ \chi_i = 1 $ in a neighborhood of $ \p \Sigma_i $ and  $ \chi_i = 0 $ in a neighborhood of $ \p B_{g_i} (p_i, \rho/2) $.
Define $\tilde g_T =  \chi_i g_i + ( 1 - \chi_i) \hat g_T $ on  $\Sigma_i \setminus  B_{g_i} (p_i, \rho/2)   $.  
Then $\tilde g_T$ converges to  $g_i$ in $C^{k, \alpha}(\overline{U}_i \setminus B_{g_i}(p_i, \rho/2))$, 
 $\mbox{vol}(\overline{U}_1 \#\overline{U}_2, \tilde g_T) \rightarrow  \mbox{vol} (\overline{U}_1, g_1)+ \mbox{vol} ( \overline{U}_2, g_2)$,  and $R(\tilde g_T)\rightarrow \sigma_n$, with $R(\tilde g_T)=\sigma_n$ in a 
 neighborhood of $\partial U_i$ and $ \partial B_{g_i}(p_i, \rho/2)$. 
Since $g_i$ is not $V$-static on $ {U}_i \setminus \overline{  B_{g_i} (p_i, \rho/2) }$, Theorem \ref{locdefTake2} can now be applied on  $ U_i \setminus \overline{ B_{g_i} (p_i, \rho/2)}$ to deform $\tilde g_T$ (for sufficiently large $ T$) to a metric $\tilde g$ such that  $( \tilde g_T  - \tilde{g} ) $ has compact support  in  ${U_i} \setminus \overline{ B_{g_i}(p_i, \rho/2) }$, $R(\tilde g)=\sigma_n$, and 
  \begin{eqnarray*} 
\mbox{vol}  (\overline{U}_1, g_1)+ \mbox{vol} ( \overline{U}_2, g_2) = 
  \mbox{vol} (\overline{U}_1 \setminus B_{g_1}(p_1, \rho/2),   \tilde g )  + \mbox{vol} (\overline{U}_2 \setminus B_{g_2}(p_2, \rho/2),  \tilde g )  \\  + \mbox{vol}(\overline{U}_1 \#\overline{U}_2 \setminus   \left(( \overline{U}_1 \setminus B_{g_1}(p_1, \rho/2)) \sqcup  ( \overline{U}_2 \setminus B_{g_2}(p_2, \rho/2) ) \right),  \hat g_T).
\end{eqnarray*}
The metric $g$ on $M_1 \# M_2$ given by  $ g = g_i $ on $M_i \setminus U_i$ , 
  $ g = \tilde g $ on  $\overline U_i \setminus \overline{  B_{g_i} (p_i, \rho/2) } $, 
  and $ g = \hat g_T $ on $(\overline{U}_1 \# \overline{U}_2 )
  \setminus  \left( ( \overline{U}_1 \setminus B_{g_1}(p_1, \rho/2)  )
 \sqcup  ( \overline{U}_2 \setminus B_{g_2}(p_2, \rho/2))\right)$  has all the properties asserted in Theorem \ref{glue}.  
 \end{proof}
 
 
\section{Counterexamples to Min-Oo's conjecture with non-trivial topology and arbitrarily large volume}

In \cite{MinOo},  Min-Oo conjectured that 
if $(\overline \Omega, g)$ is a compact Riemannian manifold with boundary 
such that $g$ has scalar curvature at least $n(n-1)$, such that
$\p \Omega$ is isometric to the standard round sphere $\mathbb{S}^{n-1}$,
and such that $ \p \Omega$ is totally geodesic in $(\overline \Omega, g)$, 
then $(\overline \Omega, g)$  is isometric to the standard round hemisphere $\mathbb{S}^n_+$. 
Various affirmative partial results under stronger hypotheses have been achieved in this direction. We refer the reader to \cite{brendle-marques-neves} for a comprehensive account of these contributions. 
Recently, Brendle, Marques, and Neves constructed a counterexample to Min-Oo's conjecture:

\begin{thm} [\protect{\cite[Theorem 7]{brendle-marques-neves}}] 
\label{thm:brendle-marques-neves} 
Given any integer $n \geq 3$, there exists a smooth metric $g$ on the hemisphere $\mathbb{S}^n_+$ with the following properties: 
\begin{enumerate}
\item The scalar curvature of $g$ is at least $n(n-1)$ at every point on $\mathbb{S}^{n}_+$. 
\item The scalar curvature of $g$ is strictly greater than $n(n-1)$ at some point on $\mathbb{S}^n_+$.
\item The metric $g$ agrees with the standard round metric on $\mathbb{S}^n_+$ in a neighborhood of the equator $\partial \mathbb{S}_+^{n}$. 
\end{enumerate}
\end{thm}

We can construct more complicated counterexamples to Min-Oo's conjecture from counterexamples such as these by combining 
Theorem \ref{locdefTake2} and Theorem \ref{thm-imp-strenthed} (see also Remark \ref{rem:alternativelyGL}): 

\begin{prop} \label{prop-Min-Oo}
Let  $ g$ be a metric on $\mathbb{S}^n_+$ that is a counterexample to Min-Oo's conjecture.  Suppose $ g $ agrees with the standard round metric $\bar g$ in a neighborhood of the equator $\partial \mathbb{S}_+^{n}$. Given any constant $ V_0 > 0 $,  there exists a counterexample $(\mathbb{S}^n_+, \hat g)$ to Min-Oo's conjecture such that 
$\mathrm{vol}(\mathbb{S}_+^{n}, \hat g) \geq V_0$.
\end{prop}

\begin{proof} By analyticity, cf. the comments following \eqref{eqn:statickappa}, the metric $g$ cannot be $V$-static on $\mathbb{S}^n_+$. 

Let $B_g(p, \rho) \subset \mathbb{S}^n_+$ be  a geodesic ball  such that $R(g) = n (n-1)$ on $B_g(p, \rho )$, 
 $\Delta_{g} + n$ has positive Dirichlet spectrum on $B_g(p, \rho )$, and $g$ is not $V$-static on $\mathbb{S}^n_+ \setminus \overline{B_g(p, \rho/2)}$. 
We can  proceed exactly as in the proof of Theorem \ref{glue} and glue $(\mathbb{S}^n_+, g)$ to a copy of itself,  first applying Theorem \ref{imp-vol} in case $\Sigma_1$ and $\Sigma_2$ are taken to be $\overline{B_g(p, \rho)}$ from each copy, and then applying Theorem \ref{locdefTake2} to each copy of the non-$V$-static region $\mathbb{S}^n_+ \setminus \overline{B_g(p, \rho/2)}$.  The resulting metric $ \hat{g}$ on $\mathbb{S}^n_+   \#  \mathbb{S}^n_+$ agrees with $g$ to infinite order at
 $\p \mathbb{S}^n_+$ in both copies of $\mathbb{S}^n_+$. Moreover, $R(\hat g)=n(n-1)$ in the neck region while $ R( \hat{g} ) = R(g) $ on $ \mathbb{S}^n_+ \setminus \overline{B_g(p, \rho/2)}$ in both copies of $\mathbb{S}^n_+$, and $ \mbox{vol} (\mathbb{S}^n_+   \#  \mathbb{S}^n_+ , \hat{g}) = 2 \mbox{vol}(\mathbb{S}^n_+, g)$. 
 Because $g$ coincides with the standard round metric $\bar g$ near $\partial \mathbb{S}^n_+$ and $\hat g$ extends smoothly to $\bar g$ across $\p \mathbb S^n_+$, we can then add a standard round hemisphere to one of the two copies of $\mathbb{S}^n_+$.  Clearly, this process can be iterated, increasing the volume at every stage by a fixed amount.   
\end{proof} 

In conjunction with Theorem \ref{thm:brendle-marques-neves}, it follows that there are counterexamples to Min-Oo's conjecture of arbitrarily large volume. In contrast, it is shown in \cite{MiaoTamMinOo} that a metric $g$ on $\mathbb{S}^n_+$ that satisfies conclusions (1) and (3) of Theorem \ref{thm:brendle-marques-neves} and which is also $C^2$-close to the standard  round metric $ \bar{g}$ on $\mathbb{S}^n_+$ has volume less than  
$\mbox{vol}(\mathbb{S}^n_+, \bar{g})$. 

\begin{remark} In the proof of Proposition \ref{prop-Min-Oo}, we can arrange that $\hat g$ agrees with the round metric $\bar g$ near $\p \mathbb S^n_+$.  Indeed, when applying the proof of Theorem \ref{locdefTake2}, we can first find a collar neighborhood $N\subset \mathbb S^n_+$ of $\p \mathbb S^n_+$, so that $B_g(p,\rho)\subset U:=\mathbb S^n_+\setminus \overline N$, and $U\setminus  \overline {B_g(p,\rho/2)}$ is not $V$-static.  We also note that the proof of Proposition \ref{prop-Min-Oo} can be applied to any counterexample of the Min-Oo conjecture which is not $V$-static and contains a domain with constant scalar curvature.  We can connect to such a space one of the examples satisfying conditions (1) and (3) of Theorem \ref{thm:brendle-marques-neves}.  After applying Theorem \ref{glue}, we can cap off this end with a round sphere as above.  
\end{remark}

\begin{remark} \label{rem:alternativelyGL} The large-volume counterexamples to Min-Oo's conjecture can alternatively be obtained from the Brendle-Marques-Neves counterexample using the Gromov-Lawson connect-sum construction for positive scalar curvature \cite{gromlaw}. The construction in \cite{gromlaw} is local near the gluing points. One would connect two copies of the example of Brendle-Marques-Neves at points where $R(g)>n(n-1)$, applying the technique of Gromov-Lawson carefully so as to maintain the lower bound on the scalar curvature. 
\end{remark}

\begin{remark}  In the proof of Proposition  \ref{prop-Min-Oo}, 
one can start with two disjoint small balls $B_g(p_1, \rho_1)$ and $B_g(p_2, \rho_2)$ in $(\mathbb{S}^n_+, g)$ 
such that  $R(g) = n (n-1)$ on $B_g(p_i, \rho_i )$, 
 $\Delta_{g} + n$ has positive Dirichlet spectrum on $B_g(p_i, \rho_i )$, 
 $i = 1, 2$, and $g$ is not $V$-static on $U = \mathbb{S}^n_+ \setminus 
 ( \overline{B_g(p_1, \frac{\rho_1}{2})} \cup \overline{B_g(p_2, \frac{\rho_2}{2} )} )$.
 By forming the connected sum of $B_g(p_1, \rho_1) $ and $B_g(p_2, \rho_2)$ (adding a handle)  and deforming the metric on $U$,  one obtains a counterexample to Min-Oo's conjecture with non-trivial fundamental group.  One can also obtain such an example by connecting a counterexample to Min-Oo's conjecture to a non-$V$-static metric on $\mathbb S^1 \times \mathbb S^{n-1}$ or $\mathbb S^n/\Gamma$ (where $\Gamma$ is a finite subgroup of $SO(n+1)$) which has scalar curvature at least $n(n-1)$, and in some region has constant scalar curvature $n(n-1)$. The existence of such metrics on $\mathbb S^1 \times \mathbb S^{n-1}$ or $\mathbb S^n/\Gamma$ follows from results of Kazdan-Warner. In fact, it is shown in \cite{kaz-war-1, kaz-war-2} that on every closed  manifold that admits a smooth metric of positive scalar curvature, every smooth function is the scalar curvature of some smooth metric. Applying the proof of Proposition \ref{prop-Min-Oo} to such an example, one  obtains more complicated counterexamples with non-trivial topology and arbitrarily large volume. 
\end{remark}


\appendix

\section{Schauder theory} \label{sec:Schauder} 

Here we discuss interior Schauder estimates in weighted spaces, following Chru\'{s}ciel and Delay \cite[Appendix B]{cd}, for the particular example of the operator $u \mapsto P(u):= \rho^{-1}L_g\rho L_g^* u$. Note that, in local coordinates, $P(u)$ has the form $$ (n-1)\Delta^2 u + \sum\limits_{|\beta|\leq 3} b_{\beta} D^{\beta} u.$$ Recall that the weight $\rho$ is a smooth ($C^{k,\alpha}$) function that behaves like $e^{- 1/d}$ near $\partial \Omega$. It is easy to check that $||b_{\beta}||_{C^{0, \alpha}_{\phi, \phi^{4 - |\beta|} }} < \infty$. By appropriate scaling, one can obtain interior Schauder estimates on small balls near the boundary of $\Omega$ from interior Schauder estimates on balls of a fixed size for an operator whose coefficients are well controlled. The weighted H\"{o}lder norms defined in Section \ref{func-sp} are designed for this purpose. 

For simplicity, we assume that we are working in $\mathbb R^n$ with the standard metric, and that $x$ is close to $\partial \Omega$ so that $\phi(x)=d(x)^2$.  For $z\in B(0,1)$, let $y= x+\phi(x) z$, and for any $f$, let $\tilde f(z)= f(x+\phi(x)z)=f(y)$.  Then $D_z\tilde u|_z = \phi(x)D_y u|_{x+\phi(x) z}$.  We compute that
\begin{align*}
(Pu) (x+\phi(x)z) &= (n-1)\Delta^2_y u \big|_{x+\phi(x) z}+ \sum\limits_{|\beta|\leq 3} b_{\beta} D_y^{\beta} u\big|_{x+\phi(x) z}\\
&= \phi(x)^{-4} (n-1)\Delta^2_z \tilde u\big|_z+ \sum\limits_{|\beta|\leq 3} \phi(x)^{-|\beta|} \tilde b_{\beta}(z) D^{\beta}_z \tilde u \big|_z.
\end{align*}
We obtain that
$$\phi(x)^4 \widetilde{Pu}(z)= \left( (n-1) \Delta^2_z+ \sum\limits_{|\beta|\leq 3} \phi(x)^{4-|\beta|} \tilde b_{\beta}(z) D_z^{\beta}\right) \tilde u =: \tilde P \tilde u (z).$$
We see that $\tilde P$ is uniformly elliptic on $B(0,1)$ and has coefficients that are bounded in $C^{0, \alpha}$ by bounds for $\|b_{\beta}\|_{C^{0, \alpha}_{\phi, \phi^{4-|\beta|}}}$.  The standard interior Schauder estimate gives 
\begin{align*}
\|\tilde u\|_{C^{4, \alpha}(B(0, \frac{1}{4}))} & \leq C \left(\|\tilde P \tilde u \|_{C^{0,\alpha}(B(0, \frac{1}{2}))} + \|\tilde u \|_{L^2(B(0, \frac{1}{2}))}\right) \\
& \leq C \left( \phi(x)^4 \| \widetilde{Pu} \|_{C^{0,\alpha}(B(0, \frac{1}{2}))} + \|\tilde u \|_{L^2(B(0, \frac{1}{2}))}\right).
\end{align*}
Scaling back, we see that 
\begin{align*}
&\sum  \limits_{j=0}^4 \phi(x)^j \| \nabla^j_g u\|_{C^{0, \alpha}(B(x, \frac{\phi(x)}{4}))}  + \phi(x)^{4+\alpha} [\nabla_g^4 u]_{0,\alpha;B(x, \frac{\phi(x)}{4})} \\
& \leq C \left( \phi(x)^4 \| Pu \|_{C^0(B(x, \frac{\phi(x)}{2}))}+\phi(x)^{4+\alpha}[Pu]_{0,\alpha;B(x, \frac{\phi(x)}{2})} +\phi(x)^{-\frac{n}{2}} \| u \|_{L^2(B(x, \frac{\phi(x)}{2}))}\right).
\end{align*}
We can multiply this inequality by $\varphi(x)$ where $\varphi=\phi^r \rho^s$ to obtain the following weighted estimate on $\Omega$:
\begin{equation}
\|u\|_{C^{4, \alpha}_{\phi, \varphi}}\leq C (\|Pu\|_{C^{0,\alpha}_{\phi, \phi^4 \varphi}}+ \|u\|_{L^2_{\phi^{-n}\varphi^2}}).  \label{wtsch0}
\end{equation}
This estimate is similar to that in \cite[Appendix B]{cd}. Note that we impose slightly different conditions on the lower order coefficients here, and that we use a different convention for the weighted $L^2$ norms.  
As for higher regularity, we obtain similarly that 
\begin{equation}
\|u\|_{C^{k, \alpha}_{\phi, \varphi}}\leq C (\|Pu\|_{C^{k-4,\alpha}_{\phi, \phi^4 \varphi}}+ \|u\|_{L^2_{\phi^{-n}\varphi^2}})  \label{wtschk}
\end{equation}
where the constant $C$ depends on the domain, the weight, and bounds for $\|b_{\beta}\|_{C^{k-4, \alpha}_{\phi, \phi^{4-|\beta|}}}$. 


\section{Proof of Proposition \ref{prop-NT}}  \label{sec:NT} Here we sketch the proof of Proposition \ref{prop-NT}, which is similar to that of \cite[Proposition 6]{imp} and \cite[Proposition 3.6]{imp:flds}.

Recall that $[-\frac{T}{2}, \frac{T}{2}]\times \mathbb S^{n-1} \cong C_T\subset \Sigma_T$.  In the proof below, H\"{o}lder norms on $C_T$ or $Q:= [-1,1]\times \mathbb S^{n-1} \subset C_T$ are indicated with an additional subscript. 

\begin{proof} We recall that  on $\Sigma_T\setminus C_T$, $\mathcal N_T(\Psi_T)=0$, and that on $[-\frac{T}{2}+1,\frac{T}{2}-1]\times \mathbb S^{n-1}$, $\Psi_T(s, \theta)=2 R^{\frac{n-2}{2}} e^{-\frac{(n-2)T}{4}} \cosh \big( \tfrac{(n-2)s}{2}\big).$
Let $\mathring \gamma= ds^2 + g_{\mathbb S^{n-1}}$ be the standard cylindrical metric.  Then $\Psi_T$ solves the equation $(-\Delta _{\mathring \gamma} + c_n R(\mathring \gamma))( \Psi_T) =0$ on $[-\frac{T}{2}+1,\frac{T}{2}-1]\times \mathbb S^{n-1}$.  Therefore,  by \eqref{metric-est} and Lemma \ref{lem-metric-decay}, we have 
\begin{align} 
\| \Delta_{\gamma_T}f-\Delta_{\mathring \gamma}f\|_{k-2, \alpha,C_T} & \lesssim e^{-T} \cosh 2s \|f\|_{k, \alpha; C_T} \label{opest}\\
\|R(\gamma_T)-R(\mathring \gamma)\|_{k-2, \alpha, C_T} & \lesssim e^{-T} \cosh 2s \nonumber. 
\end{align}  
On $Q\cong [-1,1]\times \mathbb S^{n-1}$,  (\ref{opest}) implies 
 $\|\mathcal N_T(\Psi_T)\|_{k-2, \alpha, Q}\lesssim e^{-\frac{(n+2)T}{4}}.$
This completes the estimate on $Q$. 

We now consider $C_T\setminus Q \cong ([-\frac{T}{2}, -1]\times \mathbb S^{n-1})\cup ([1, \frac{T}{2}]\times \mathbb S^{n-1})$.   The estimates on the two components are similar. We will do one of them.  

Recall that on $[1, \frac{T}{2}]\times \mathbb S^{n-1}$ we have that $\gamma_T= \Psi_2^{-{4/(n-2)}} \gamma_2$ and $\Psi_T(s, \theta)=\chi_{1,T}(s)\psi (R e^{-s- \frac{T}{2}})+ \chi_{2,T}(s) \psi( R e^{s - \frac{T}{2}})$. Moreover, $\mathcal N_T(\Psi_2)=0$ in this region, so that 
$$\mathcal N_T(\Psi_T)= (-\Delta_{\gamma_T} +c_n R(\gamma_T))(\chi_{1,T} \Psi_1) -c_n \sigma_n \Psi_T^{\frac{n+2}{n-2}} +c_n\sigma_n \Psi_2^{\frac{n+2}{n-2}}.$$ 
We write the last two terms using $\Psi_T^{\frac{n+2}{n-2}}- \Psi_2^{\frac{n+2}{n-2}}=  \Psi_2^{\frac{n+2}{n-2}} \left( \Big ( 1+ \chi_{1,T} \frac{ \Psi_1}{\Psi_2}\Big)^{\frac{n+2}{n-2}} - 1\right).$ Since also $\frac{\Psi_1}{\Psi_2}= e^{-s(n-2)}$ and $\Psi_2^{\frac{n+2}{n-2}}= (Re^{s-\frac{T}{2}})^{\frac{n+2}{2}}$ in this region, we obtain that 
\begin{align}
\left| \Psi_T^{\frac{n+2}{n-2}}- \Psi_2^{\frac{n+2}{n-2}}  \right|\lesssim \Psi_2^{\frac{n+2}{n-2}}\cdot \frac{\Psi_1}{\Psi_2} & \lesssim e^{-s\frac{n-6}{2}} e^{-\frac{(n+2)T}{4}}. \label{nlest}
\end{align}
On $[1, \frac{T}{2}-1]\times \mathbb S^{n-1}$,   (\ref{nlest}) shows
\begin{align*}
\Big| \Psi_T^{\frac{n+2}{n-2}}  & - \Psi_2^{\frac{n+2}{n-2}}  \Big|
\lesssim \begin{cases} e^{-\frac{(n+2)T}{4}}   \qquad \qquad \textrm{for }  n > 6 \\ 
e^{-\frac{(n-2)T}{2}}  \qquad \qquad \textrm{for }  n \leq 6 \; \end{cases} 
 \end{align*}
while  on  $[\frac{T}{2}-1, \frac{T}{2}]\times \mathbb S^{n-1}$, (\ref{nlest}) gives that
$$\left| \Psi_T^{\frac{n+2}{n-2}}- \Psi_2^{\frac{n+2}{n-2}}  \right|\lesssim  e^{-\frac{(n-2)T}{2}}.$$
The required H\"{o}lder bounds of $\|\Psi_T^{\frac{n+2}{n-2}}- \Psi_2^{\frac{n+2}{n-2}} \|_{k-2,\alpha,C_T}$ follow analogously. 

It remains to estimate  $\|(-\Delta_{\gamma_T} + c_n R(\gamma_T))(\chi_{1,T}\Psi_1)\|_{k-2,\alpha,C_T}$.   
We first estimate on $[1, \frac{T}{2}-1]\times \mathbb S^{n-1}$, where  $\chi_{1,T}=1$.  Using this along with (\ref{opest}) and the fact that $\Psi_1$ is in the kernel of the conformal Laplacian on the cylinder, we obtain  
\begin{align*}
\Big|(-\Delta_{\gamma_T} +  & c_n R(\gamma_T))(\chi_{1,T}\Psi_1) \Big| \\ &= \Big|  (-\Delta_{\gamma_T} + c_n R(\gamma_T))(\Psi_1)- (-\Delta_{\mathring \gamma} + c_n R(\mathring \gamma))(\Psi_1)   \Big| \\& \lesssim e^{-T} \cosh (2s)\|\Psi_1\|_{C^2(C_T)}\\
 &\lesssim e^{s(2-\frac{n-2}{2})}e^{-T}e^{-\frac{(n-2)T}{4}}\\
 & \lesssim \begin{cases} e^{-T}e^{-\frac{(n-2)T}{4}} = e^{ - \frac{(n+2)}{4} T } \qquad \qquad \;\;\;\;\;\textrm{for } n > 6  \\e^{(1-\frac{n-2}{4})T} e^{-T}e^{-\frac{(n-2)T}{4}}= e^{-\frac{(n-2)T}{2}} \quad\;  \textrm{for }  n \leq 6  . \end{cases}
 \end{align*}
 For $s\in [\frac{T}{2}-1, \frac{T}{2}]$,  using that $\Psi_1=   (Re^{-s})^{\frac{n-2}{2}}e^{-\frac{(n-2)T}{4}}$, we have 
\begin{align*}
\Big| & (-\Delta_{\gamma_T} +  c_n R(\gamma_T))(\chi_{1,T}\Psi_1) \Big| \\  & \lesssim e^{-T} \cosh (2s) \|\Psi_1\|_{C^2(C_T)} +  \Big| (-\Delta_{\mathring \gamma} + c_n R(\mathring \gamma))(\chi_{1,T}\Psi_1) \Big|
 \lesssim e^{-\frac{(n-2)T}{2}} \; . 
 \end{align*}
 This proves the desired $C^0$ bound in  \eqref{approxsol1} and  \eqref{approxsol2}.
 The estimate of the derivatives and the H\"{o}lder bound follow from similar reasoning.  \end{proof}


\begin{thebibliography}{99}
\bibitem{be} Besse, A.: \emph{Einstein Manifolds}.  Berlin: Springer-Verlag, 1987
\bibitem{bray-lee} Bray, H., Lee, D. A.: On the Riemannian Penrose inequality in dimensions less than eight.  Duke Math. J. \textbf{148}, no. 1, 81-106 (2009)
\bibitem{brendle-marques-neves} Brendle, S., Marques, F., Neves, A.: Deformations for the hemisphere that increase the scalar curvature. Invent. Math. \textbf{185}, no. 1, 175-197 (2010)
\bibitem{cd} Chru\'{s}ciel, P. T., Delay, E.: On mapping properties of the general relativistic constraints operator in weighted function spaces, with applications. M\'{e}m. Soc. Math. Fr. (N.S.) \textbf{94} (2003)
\bibitem{cip} Chru\'{s}ciel, P. T., Isenberg, J., Pollack, D.: Initial Data Engineering.  Comm. Math. Phys. \textbf{257}, no. 1, 29--42 (2005)
\bibitem{cpp} Chru\'{s}ciel, P. T., Pacard, F., Pollack, D.: Singular Yamabe metrics and initial data with \emph{exactly} Kotter-Schwarzschild-de Sitter ends. II. Generic metrics. Math. Res. Lett. \textbf{16}, no. 1, 157-164 (2009)
\bibitem{cor:schw} Corvino, J.: Scalar curvature deformation and a gluing construction for the Einstein constraint equations. Comm. Math. Phys. \textbf{214}, 137--189 (2000)
\bibitem{corpol} Corvino, J., Pollack, D.: Scalar curvature and the Einstein constraint equations. In: Bray, H. L., Minicozzi, W. P., II (eds.), \emph{Surveys in Geometric Analysis and Relativity}. Adv. Lect. Math. \textbf{20}, pp. 145-188 (2011)
\bibitem{cs:ak} Corvino, J., Schoen, R. M.: On the asymptotics of the vacuum Einstein constraint equations. J. Differential Geom. \textbf{73}, no. 2, 185--217 (2006)
\bibitem{delay} Delay, E.: Localized gluing of Riemannian metrics in interpolating their scalar curvature.  Differential Geom. Appl. \textbf{29}, no. 3, 433-439 (2011)
\bibitem{Foote} Foote, R.: Regularity of the distance function. Proc. Amer. Math. Soc. \textbf{92}, no. 1, 153-155 (1984). 
\bibitem{gromlaw}  Gromov, M., Lawson, H. B., Jr.: The classification of simply connected manifolds of positive scalar curvature.  Ann. of Math. (2) \textbf{111}, no. 3, 423-434 (1980)
\bibitem{imp} Isenberg, J., Mazzeo, R., Pollack, D.: Gluing and wormholes for the Einstein constraint equations. Comm. Math. Phys. \textbf{231}, no. 3, 529--568 (2002)
\bibitem{imp:flds} Isenberg, J., Maxwell, D., Pollack, D.: A gluing construction for non-vacuum solutions of the Einstein constraint equations. Adv. Theor. Math. Phys. \textbf{9}, no. 1, 129-172 (2005)
\bibitem{joyce} Joyce, D.: Constant scalar curvature metrics on connected sums.  Int. J. Math. Math. Sci., no. 7, 405-450 (2003) 
\bibitem{kaz-war-1} Kazdan, J. L., Warner, F. W.: Existence and conformal deformation of metrics with prescribed Gaussian and scalar curvature. Ann. Math. (2) \textbf{101}, 317-331 (1975)
\bibitem{kaz-war-2} Kazdan, J. L., Warner, F. W.: A direct approach to the determination of Gaussian and scalar curvature functions. Invent. Math. \textbf{28}, 227-230 (1975)
\bibitem{mpu} Mazzeo, R., Pollack, D., Uhlenbeck, K.: Connected sum constructions for constant scalar curvature metrics. Topol. Methods Nonlinear Anal. \textbf{6}, no. 2, 207-233 (1995)
\bibitem{MiaoShiTam09} Miao, P.,  Shi, Y.-G. and  Tam, L.-F.: On geometric problems related 
to Brown-York and Liu-Yau quasilocal mass.  Comm. Math. Phys. \textbf{298}, no. 2, 437--459 (2010)
\bibitem{MiaoTam08} Miao, P., Tam, L.-F.: On the volume functional of compact manifolds with boundary with constant scalar curvature. Calc. Var. Partial Differential Equations.  \textbf{36}, no. 2, 141--171 (2009) 
\bibitem{MiaoTam-TAMS} Miao, P. and Tam, L.-F.: Einstein and conformally flat critical metrics of the volume functional. Trans. Amer. Math. Soc.
\textbf{363}, 2907-2937 (2011) 
\bibitem{MiaoTamMinOo} Miao, P. and Tam, L.-F.: Scalar curvature rigidity with a volume constraint.  Comm. Anal. Geom.  \textbf{20}, no. 1, 1-30 (2012)
\bibitem{MinOo} Min-Oo, M.: Scalar curvature rigidity of certain symmetric spaces. In: Lalonde, F. (ed.) \emph{Geometry, topology, and dynamics (Montreal, 1995)}, CRM Proc. Lecture Notes (Amer. Math. Soc., Providence RI) \textbf{15}, pp. 127-136 (1998)
\bibitem{ob}  Obata, M.: Certain conditions for a Riemannian manifold to be isometric with a sphere.  J. Math. Soc. Japan \textbf{14}, 333--340 (1962) 
\bibitem{singyam} Schoen, R. M.: The existence of weak solutions with prescribed singular behavior for a conformally invariant scalar equation. Comm. Pure and Appl. Math. \textbf{41}, 317-392 (1988)
\bibitem{SchoenYau79} Schoen, R. and Yau, S.-T.: On the proof of the positive mass conjecture in general relativity. Comm. Math. Phys. \textbf{65}, 45--76 (1979)
\bibitem{sy:psc} Schoen, R. M., Yau, S.-T.: On the structure of manifolds with positive scalar curvature. Manuscripta Math. \textbf{28}, no. 1-3, 159-183 (1979) 
\bibitem{Schneider}
Schneider, R., {\it Convex Bodies: The Brunn-Minkowski Theory}. 
Encyclopedia of Mathematics and its Applications, \textbf{44}.
Cambridge University Press,  1993.
\bibitem{sh} Shen, Y.: A note on Fischer-Marsden's conjecture.  Proc. Amer. Math. Soc. \textbf{125}, no. 3, 901-905 (1997) 
\bibitem{ShiTam02}
Shi, Y.-G. and  Tam, L.-F.: Positive mass theorem
 and the boundary behaviors of compact manifolds with nonnegative scalar curvature. 
J. Differential Geom. \textbf{62}, 79--125 (2002)
\bibitem{Tam-private}
Tam, L.-F., 
Private communication (2010)
\bibitem{Witten81} Witten, E.: A new proof of the positive energy theorem.  Comm. Math. Phys. \textbf{80}, 381--402 (1981)
\end{thebibliography}
\end{document}